\documentclass[12pt]{amsart}
\usepackage{xspace,amssymb,amsfonts,epsfig,marvosym,euscript,eufrak,xypic,enumerate,amsmath,nicefrac}
\usepackage{tikz,etex,bbm,xspace}
 \usepackage{epstopdf,ifpdf,hyperref}
\usepackage{pxfonts,multicol}
\usepackage[margin=3cm,footskip=25pt,headheight=20pt]{geometry}
\usepackage{fancyhdr,mathrsfs}
\usepackage[polutonikogreek,english]{babel}
\ifpdf
  \usepackage{pdfsync}
\fi

\begin{document}
\newtheoremstyle{all}
  {11pt}
  {11pt}
  {\slshape}
  {}
  {\bfseries}
  {}
  {.5em}
  {}

\theoremstyle{all}
\newtheorem{thm}{Theorem}[section]
\newtheorem{prop}[thm]{Proposition}
\newtheorem{cor}[thm]{Corollary}
\newtheorem{lemma}[thm]{Lemma}
\newtheorem{defn}[thm]{Definition}
\newtheorem{ques}[thm]{Question}
\newtheorem{conj}[thm]{Conjecture}
\newtheorem{rem}[thm]{Remark}
 
\newcommand{\nc}{\newcommand}
\newcommand{\renc}{\renewcommand}
  \nc{\kac}{\kappa^C}
\nc{\alg}{T}
\nc{\Lco}{L_{\la}}
\nc{\qD}{q^{\nicefrac 1D}}
\nc{\ocL}{M_{\la}}
\nc{\excise}[1]{}
\nc{\Dbe}{D^{\uparrow}}
\nc{\tr}{\operatorname{tr}}
\newcommand{\Mirkovic}{Mirkovi\'c\xspace}
\nc{\tla}{\mathsf{t}_\la}
\nc{\llrr}{\langle\la,\rho\rangle}
\nc{\lllr}{\langle\la,\la\rangle}
\nc{\K}{\mathbbm{k}}
\nc{\Stosic}{Sto{\v{s}}i{\'c}\xspace}
\newcommand{\arxiv}[1]{\href{http://arxiv.org/abs/#1}{\tt arXiv:\nolinkurl{#1}}}
\nc{\cd}{\mathcal{D}}
\nc{\vd}{\mathbb{D}}
\nc{\R}{\mathbb{R}}
\renewcommand{\theequation}{\fnsymbol{equation}}

\renc{\wr}{\operatorname{wr}}
  \nc{\Lam}[3]{\La^{#1}_{#2,#3}}
  \nc{\Lab}[2]{\La^{#1}_{#2}}
  \nc{\Lamvwy}{\Lam\Bv\Bw\By}
  \nc{\Labwv}{\Lab\Bw\Bv}
  \nc{\nak}[3]{\mathcal{N}(#1,#2,#3)}
  \nc{\hw}{highest weight\xspace}
  \nc{\al}{\alpha}
  \nc{\be}{\beta}
  \nc{\bM}{\mathbf{m}}
  \nc{\bkh}{\backslash}
  \nc{\Bi}{\mathbf{i}}
  \nc{\Bj}{\mathbf{j}}
\nc{\RAA}{R^\A_A}
  \nc{\Bv}{\mathbf{v}}
  \nc{\Bw}{\mathbf{w}}
\nc{\Id}{\operatorname{Id}}
  \nc{\By}{\mathbf{y}}
\nc{\eE}{\EuScript{E}}
  \nc{\Bz}{\mathbf{z}}
  \nc{\coker}{\mathrm{coker}\,}
  \nc{\C}{\mathbb{C}}
  \nc{\ch}{\mathrm{ch}}
  \nc{\de}{\delta}
  \nc{\ep}{\epsilon}
  \nc{\Rep}[2]{\mathsf{Rep}_{#1}^{#2}}
  \nc{\Ev}[2]{E_{#1}^{#2}}
  \nc{\fr}[1]{\mathfrak{#1}}
  \nc{\fp}{\fr p}
  \nc{\fq}{\fr q}
  \nc{\fl}{\fr l}
  \nc{\fgl}{\fr{gl}}
\nc{\rad}{\operatorname{rad}}
\nc{\ind}{\operatorname{ind}}
  \nc{\GL}{\mathrm{GL}}
  \nc{\Hom}{\mathrm{Hom}}
  \nc{\im}{\mathrm{im}\,}
  \nc{\La}{\Lambda}
  \nc{\la}{\lambda}
  \nc{\mult}{b^{\mu}_{\la_0}\!}
  \nc{\mc}[1]{\mathcal{#1}}
  \nc{\om}{\omega}
\nc{\gl}{\mathfrak{gl}}
  \nc{\cF}{\mathcal{F}}
 \nc{\cC}{\mathcal{C}}
  \nc{\Vect}{\mathsf{Vect}}
 \nc{\modu}{\mathsf{mod}}
  \nc{\qvw}[1]{\La(#1 \Bv,\Bw)}
  \nc{\van}[1]{\nu_{#1}}
  \nc{\Rperp}{R^\vee(X_0)^{\perp}}
  \nc{\si}{\sigma}
  \nc{\croot}[1]{\al^\vee_{#1}}
\nc{\di}{\mathbf{d}}
  \nc{\SL}[1]{\mathrm{SL}_{#1}}
  \nc{\Th}{\theta}
  \nc{\vp}{\varphi}
  \nc{\wt}{\mathrm{wt}}
  \nc{\Z}{\mathbb{Z}}
  \nc{\Znn}{\Z_{\geq 0}}
  \nc{\ver}{\EuScript{V}}
  \nc{\Res}[2]{\operatorname{Res}^{#1}_{#2}}
  \nc{\edge}{\EuScript{E}}
  \nc{\Spec}{\mathrm{Spec}}
  \nc{\tie}{\EuScript{T}}
  \nc{\ml}[1]{\mathbb{D}^{#1}}
  \nc{\fQ}{\mathfrak{Q}}
        \nc{\fg}{\mathfrak{g}}
  \nc{\Uq}{U_q(\fg)}
        \nc{\bom}{\boldsymbol{\omega}}
\nc{\bla}{{\underline{\boldsymbol{\la}}}}
\nc{\bmu}{{\underline{\boldsymbol{\mu}}}}
\nc{\bal}{{\boldsymbol{\al}}}
\nc{\bet}{{\boldsymbol{\eta}}}
\nc{\rola}{X}
\nc{\wela}{Y}
\nc{\fM}{\mathfrak{M}}
\nc{\fX}{\mathfrak{X}}
\nc{\fH}{\mathfrak{H}}
\nc{\fE}{\mathfrak{E}}
\nc{\fF}{\mathfrak{F}}
\nc{\fI}{\mathfrak{I}}
\nc{\qui}[2]{\fM_{#1}^{#2}}
\renc{\cL}{\mathcal{L}}
\nc{\ca}[2]{\fQ_{#1}^{#2}}
 \nc{\cat}{\mathcal{V}}
\nc{\cata}{\mathfrak{V}}
\nc{\pil}{{\boldsymbol{\pi}}^L}
\nc{\pir}{{\boldsymbol{\pi}}^R}
\nc{\cO}{\mathcal{O}}
\nc{\Ko}{\text{\Denarius}}
\nc{\Ei}{\fE_i}
\nc{\Fi}{\fF_i}
\nc{\fil}{\mathcal{H}}
\nc{\brr}[2]{\beta^R_{#1,#2}}
\nc{\brl}[2]{\beta^L_{#1,#2}}
\nc{\so}[2]{\EuScript{Q}^{#1}_{#2}}
\nc{\EW}{\mathbf{W}}
\nc{\rma}[2]{\mathbf{R}_{#1,#2}}
\nc{\Dif}{\EuScript{D}}
\nc{\MDif}{\EuScript{E}}
\renc{\mod}{\mathsf{mod}}
\nc{\modg}{\mathsf{mod}^g}
\nc{\fmod}{\mathsf{mod}^{fd}}
\nc{\id}{\operatorname{id}}
\nc{\DR}{\mathbf{DR}}
\nc{\End}{\operatorname{End}}
\nc{\Fun}{\operatorname{Fun}}
\nc{\Ext}{\operatorname{Ext}}
\nc{\tw}{\tau}
\nc{\A}{\EuScript{A}}
\nc{\Loc}{\mathsf{Loc}}
\nc{\eF}{\EuScript{F}}
\nc{\LAA}{\Loc^{\A}_{A}}
\nc{\perv}{\mathsf{Perv}}
\nc{\gfq}[2]{B_{#1}^{#2}}
\nc{\qgf}[1]{A_{#1}}
\nc{\qgr}{\qgf\rho}
\nc{\tqgf}{\tilde A}
\nc{\Tr}{\operatorname{Tr}}
\nc{\Tor}{\operatorname{Tor}}
\nc{\cQ}{\mathcal{Q}}
\nc{\st}[1]{\Delta(#1)}
\nc{\cst}[1]{\nabla(#1)}
\nc{\ei}{\mathbf{e}_i}
\nc{\Be}{\mathbf{e}}
\nc{\Hck}{\mathfrak{H}}
\renc{\P}{\mathbb{P}}
\nc{\cI}{\mathcal{I}}
\nc{\coe}{\mathfrak{K}}
\nc{\pr}{\operatorname{pr}}
\nc{\bra}{\mathfrak{B}}
\nc{\rcl}{\rho^\vee(\la)}
\nc{\tU}{\mathcal{U}}
\nc{\RHom}{\mathrm{RHom}}
\nc{\tcO}{\tilde{\cO}}

\setcounter{tocdepth}{1}

\excise{
\newenvironment{block}
\newenvironment{frame}
\newenvironment{tikzpicture}
\newenvironment{equation*}
}

\baselineskip=1.1\baselineskip

 \usetikzlibrary{decorations.pathreplacing,backgrounds,decorations.markings}
\tikzset{wei/.style={draw=red,double=red!40!white,double distance=1.5pt,thin}}
\tikzset{bdot/.style={fill,circle,color=blue,inner sep=3pt,outer sep=0}}
\begin{center}
\noindent {\large  \bf Knot invariants and higher representation theory II:\\ the categorification of quantum knot invariants}
\medskip

\noindent {\sc Ben Webster}\footnote{Supported by an NSF Postdoctoral Research Fellowship and  by the NSA under Grant H98230-10-1-0199.}\\  
Department of Mathematics\\ Northeastern University\\
Boston, MA\\
Email: {\tt b.webster@neu.edu}
\end{center}
\bigskip
{\small
\begin{quote}
\noindent {\em Abstract.}
We construct knot invariants categorifying the quantum knot variants for all representations of quantum groups.  We show that these invariants coincide with previous invariants defined by Khovanov for $\mathfrak{sl}_2$ and $\mathfrak{sl}_3$ and by Mazorchuk-Stroppel and Sussan for $\mathfrak{sl}_n$.   

Our technique uses categorifications of the tensor product
representations of Kac-Moody algebras and quantum groups, constructed
in part I of this paper. These categories are based on the pictorial
approach of Khovanov and Lauda. In this paper, we show that these
categories are related by functors corresponding to the braiding and
(co)evaluation maps between representations of quantum groups.
Exactly as these maps can be used to define quantum invariants
attached to any tangle, their categorifications can be used to define
knot homologies.
\end{quote}
}

\vspace{1cm}

\renc{\thethm}{\Alph{thm}}

Much of the theory of quantum topology rests on the structure of
monoidal categories and their use in a variety of topological
constructions. In this paper, we define a categorification of one
of these: the R-matrix construction of quantum knot invariants,
following Reshetikhin and Turaev \cite{Tur,RT}. 

They
construct polynomial invariants of framed knots by assigning natural maps
between tensor products of representations of a quantized universal
enveloping algebra $U_q(\fg)$ to each ribbon tangle labeled with
representations.  These maps are natural with respect to tangle
composition; thus they can be reconstructed from a small number of
constituents, most notably the maps associated to a single ribbon
twist, single crossing, single cup and single cap.
The map associated to a link whose components are
labeled with a representation of $\fg$ (or the corresponding highest
weight) is thus simply a Laurent polynomial.

Particular cases of these include:
\begin{itemize}
\item the {\bf Jones polynomial} when $\fg=\mathfrak{sl}_2$ and all strands
  are labeled with the defining representation.
\item the {\bf colored Jones polynomials} for other representations of
  $\fg=\mathfrak{sl}_2$.
\item specializations of the {\bf HOMFLYPT polynomial} for the defining
  representation of $\fg=\mathfrak{sl}_n$.
\item the {\bf Kauffman polynomial} (not to be confused with the Kauffman
  bracket, a variant of the Jones polynomial) for the defining
  representation of $\mathfrak{so}_{n}$.
\end{itemize}

These special cases have been categorified to knot homologies from a
number of perspectives by Khovanov and Khovanov-Rozansky
\cite{Kho00,Kho02,Kh-sl3,Kho05,KR04,KRSO,KR05}, Stroppel and
Mazorchuk-Stroppel \cite{StDuke,MS09}, Sussan \cite{Sussan2007},
Seidel-Smith \cite{SS}, Manolescu \cite{Manolescu}, Cautis-Kamnitzer
\cite{CK,CKII}, Mackaay, \Stosic and Vaz \cite{MSVsln,MSV} and the
author and Williamson \cite{WWcol}. However all of these have only
considered minuscule representations (of which there are only finitely
many in each type).

There has been some progress on other representations of
$\mathfrak{sl}_2$.  In a paper still in preparation,  Stroppel and Sussan also consider
the case of the colored Jones polynomial \cite{StSu} (building on
previous work with Frenkel \cite{FSS}); it seems likely their
construction is equivalent to ours via the constructions of Section
\ref{sec:comparison-functors}. Similarly, Cooper, Hogancamp and
Krushkal have given a categorification of the 2-colored Jones
polynomial in Bar-Natan's cobordism formalism for Khovanov homology \cite{CoHoKr}.

On the other hand, the work of physicists suggests that
categorifications for all representations exist; one schema for
defining them is given by Witten \cite{Witknots}.  The relationship
between these invariants arising from gauge theories and those
presented in this paper is completely unknown (at least to the author)
and presents a very interesting question for consideration in the future.\medskip

However, the vast majority of representations previously
had no homology theory attached to them.  In this paper, we will
construct such a theory for any labels; that is,
\begin{thm}\label{knot-hom}
  For each simple complex Lie algebra $\fg$, there is a
  homology theory $\EuScript{K}(L,\{\la_i\})$ for links $L$ whose
  components are labeled by finite dimensional representations of
  $\fg$ (here indicated by their highest weights $\la_i$), which
  associates to such a link a bigraded vector space whose graded Euler
  characteristic is the quantum invariant of this labeled link.

  This theory coincides up to grading shift with Khovanov's homologies for
  $\fg=\mathfrak{sl}_2,\mathfrak{sl}_3$ when the link is labeled with the defining representation of these algebras, and the
  Mazorchuk-Stroppel-Sussan homology for the defining representation
  of $\mathfrak{sl}_n$.
\end{thm}
Conjecturally, the Mazorchuk-Stroppel-Sussan homology is canonically
isomorphic to Khovanov-Rozansky homology (see \cite[\S 7]{MS09}); they both categorify the same knot invariants.

At the moment, we have not proven that this theory is functorial, but
we do have a proposal for the map associated to a cobordism when the weights $\la_i$ are all minuscule.  As usual
in knot homology, this proposed functoriality map is constructed by picking a
Morse function on the cobordism, and associating simple maps to the
addition of handles.  At the moment, we have no proof that this definition is
independent of Morse function and we anticipate that proving this will be quite difficult.

\medskip

Our method for this construction is to categorify every structure on
the ribbon category of $U_q(\fg)$-representations used in the original
definition: its braiding, ribbon structure, and rigid structure (the
functor of taking duals).  This approach was pioneered by Stroppel for
the defining rep of $\mathfrak{sl}_2$ \cite{Str06b,Str06a} and was
extended to $\mathfrak{sl}_n$ by Sussan \cite{Sussan2007} and
Mazorchuk-Stroppel \cite{MS09}.  But for our approach, we must use
much less familiar categories than the variations of category $\cO$
used by those authors. These categories are introduced by the author in \cite{WebCTP}, and our primary task in this paper to
construct and check relations between functors analogous to the
translation and twisting functors that appear in the $\mathfrak{sl}_n$
case (which our construction will specialize to).  

The principal result of  \cite{WebCTP} is that for each ordered $\ell$-tuple $\bla=(\la_1,\dots,\la_\ell)$ of
  dominant weights of $\fg$, there is a graded finite dimensional algebra
  $\alg^\bla$ whose representations are a module category for the categorification of $U_q(\fg)$ in the sense of Rouquier and Khovanov-Lauda and whose graded Grothendieck
  group $K_0(\alg^\bla)$ is an integral form of the $U_q(\fg)$-representation $V_{\la_1}\otimes \cdots\otimes V_{\la_\ell}$. 

In this paper, we strengthen the case for viewing $\cata^\bla$, the category of finite dimen\-sional $\alg^\bla$-modules and its derived category $\cat^\bla=\Dbe(\cata^\bla)$ as categorifications of tensor products of $U_q(\fg)$-modules: 

\begin{thm}\label{categor}

 The derived category $\cat^\bla$  carries functors categorifying all the structure maps of the ribbon category of $U_q(\fg)$-modules:
\begin{enumerate} 
\renc{\labelenumi}{(\roman{enumi})}
\item If $\si$ is a braid, then we have an exact functor
  $\mathbb{B}_{\si}\colon\cat^\bla\to\cat^{\si(\bla)}$
  such that the induced map $K_0(\alg^\bla)\to
  K_0(\alg^{\si(\bla)})$ is the action of the appropriate
  composition of R-matrices and flips.  Furthermore, these functors
  induce a strong action of the braid groupoid on the categories
  associated to permutations of the set $\bla$.
\item If two consecutive elements of $\bla$ label dual representations
  and $\bla^-$ denotes the sequence with these removed, then there are functors $\mathbb{T},\mathbb{E}\colon \cat^{\bla}\to \cat^{\bla^-}$
  which induces the quantum trace and evaluation on the Grothendieck group, and
  similarly functors $\mathbb{K},\mathbb{C}\colon \cat^{\bla^-}\to
  \cat^{\bla}$ for the coevaluation map and quantum cotrace maps.
\item When $\fg=\mathfrak{sl}_n$,  the structure functors above can be described in
  terms of twisting and Enright-Shelton functors on $\cO$.
\end{enumerate}
\end{thm}

As mentioned earlier, these functors have a
topological interpretation: the algebra $\alg^\bla$ is defined using red strands labeled with weights; we imagine placing them in $\R^3$ and thickening them to ribbons (so
that we keep track of twists in them). Then our functors correspond to
the following operations on ribbons:
\begin{itemize}
\item Crossing two ribbons: the corresponding operator in
  representations of the quantum group is called the {\bf braiding} or
  {\bf R-matrix}\footnote{As usual, the R-matrix is a map between tensor products of representations $V\otimes W\to V\otimes W$ intertwining the usual and opposite coproducts; we use the term {\bf braiding} to refer to the composition of this with the usual flip map, which is thus a homomorphism of representations $V\otimes W\to W\otimes V$.}.
\item Creating a cup, or closing a cap: the corresponding operators in
  representations of the quantum group are called the {\bf
    coevaluation} and {\bf quantum trace}.
\item Adding a full twist to one of the ribbons: the corresponding
  operator in the quantum group is called the {\bf ribbon element}.
\end{itemize}

Since all ribbon knots can be built using these operations, the quantum knot invariants are given by a composition of the
decategorifications of the functors constructed in Theorem
\ref{categor}, as described in \cite[\S 4]{CP}; combining the functors
themselves in the same pattern gives the knot homology of Theorem
\ref{knot-hom}.
\medskip

Let us now summarize the structure of the paper.
\begin{itemize}
\item In Section \ref{sec:braid-rigid}, we prove Theorem
  \ref{categor}(i).  That is, we construct the functor lifting the
  braiding of the monoidal category of
  $U_q(\fg)$-representations.  This functor is derived tensor product
  with a natural bimodule.  A particularly interesting and important
  special case is the functors corresponding to the half-twist braid,
  which sends projective modules to tiltings and the full twist braid,
  which we show gives the right Serre functor of $\cat^\bla$. 
\item In Section \ref{sec:rigid}, we prove Theorem
  \ref{categor}(ii). The most important element of this is to identify
  a special simple module in the category for a pair of dual
  fundamental weights, which categorifies an invariant vector.
  Interestingly, we are essentially forced to choose a non-standard
  ribbon element in order to obtain a ribbon functor which fits the
  same compatibilities.  This means we will categorify the knot
  invariants for a slightly unusual ribbon structure on the category
  of $U_q(\fg)$ modules, but this will only have the effect of
  multiplying the quantum invariants by an easily determined sign (see
  Proposition \ref{schur-indicate}).
\item In Section \ref{sec:invariants}, we prove Theorem \ref{knot-hom}
  using the functors constructed in Theorem \ref{categor} and a small
  number of explicit computations.  We also suggest a map for the
  functoriality along a cobordism between links. However, this map is 
  defined by choosing a handle decomposition of the cobordism, and at
  the moment, we have no proof that the induced map is independent of
  this choice.
\item In Section \ref{sec:comparison-functors}, we consider the special case where $\fg\cong \mathfrak{sl}_n$; in our previous paper \cite{WebCTP}, we showed that the categories in this case are related to category $\cO$ for $\mathfrak{gl}_N$.  Now we relate the functors  appearing Theorem
  \ref{categor} to previously defined functors on category $\cO$.

  This allows us to show the portions of Theorem \ref{knot-hom} regarding comparisons to Khovanov homology and
  Mazorchuk-Stroppel-Sussan homology.
\end{itemize}

  We should note that an earlier version of this paper
  had contained some results relating on canonical bases.  In the
  interest of giving these results in sufficient detail, they have been moved to another paper \cite{WebCB}.

\subsection*{Notation}

We let $\fg$ be a finite-dimensional simple complex Lie algebra, which
we will assume is fixed for the remainder of the paper.  
In \cite{WebCB}, we
will investigate tensor products of highest and lowest weight modules
for arbitrary symmetrizable Kac-Moody algebras, hopefully allowing us
to extend the contents of Sections \ref{sec:braid-rigid}, \ref{sec:rigid} and
\ref{sec:invariants} to this case.

We fix from now on an order on the simple roots of $\fg$, which we
will simply denote with $i<j$ for two nodes $i,j$.  This choice is
purely auxiliary, but will be useful for breaking symmetries.

Consider the weight
lattice $\wela(\fg)$ and root lattice $\rola(\fg)$, and the
simple roots $\al_i$ and coroots $\al_i^\vee$.  Let
$c_{ij}=\al_j^{\vee}(\al_i)$ be the entries of the Cartan matrix.  Let $D$ be the determinant of the Cartan matrix.  For technical reasons, it will often be convenient for us to adjoint a $D$th root of $q$, which we denote $q^{\nicefrac 1D}$.

We let $\langle
-,-\rangle$ denote the symmetrized inner product on $\wela(\fg)$,
fixed by the fact that the shortest root has length $\sqrt{2}$
and $$2\frac{\langle \al_i,\la\rangle}{\langle
  \al_i,\al_i\rangle}=\al_i^\vee(\la).$$ As usual, we let $2d_i
=\langle\al_i,\al_i\rangle$, and for $\la\in\wela(\fg)$, we
let $$\la^i=\al_i^\vee(\la)=\langle\al_i,\la\rangle/d_i.$$

We let $\rho$ be the unique weight such that $\al^\vee_i(\rho)=1$ for all $i$ and
$\rho^\vee$ the unique coweight such that $\rho^\vee(\al_i)=1$ for all $i$.  We
note that since $\rho\in \nicefrac 12\rola$ and $\rho^\vee\in
\nicefrac 12\wela^*$, for any weight $\la$, the numbers $\llrr$ and
$\rcl$ are not necessarily integers, but $2\llrr$ and $2\rcl$
are (not necessarily even) integers.

Throughout the paper, we will use $\bla=(\la_1,\dots, \la_\ell)$ to
denote an ordered $\ell$-tuple of dominant weights, and always use the
notation $\la=\sum_{i}\la_i$.  

We let $U_q(\fg)$ denote the deformed universal enveloping algebra of
$\fg$; that is, the 
associative $\C(\qD)$-algebra given by generators $E_i$, $F_i$, $K_{\mu}$ for $i$ and $\mu \in \wela(\fg)$, subject to the relations:
\begin{center}
\begin{enumerate}[i)]
 \item $K_0=1$, $K_{\mu}K_{\mu'}=K_{\mu+\mu'}$ for all $\mu,\mu' \in \wela(\fg)$,
 \item $K_{\mu}E_i = q^{\al_i^{\vee}(\mu)}E_iK_{\mu}$ for all $\mu \in
 \wela(\fg)$,
 \item $K_{\mu}F_i = q^{ -\al_i^{\vee}(\mu)}F_iK_{\mu}$ for all $\mu \in
 \wela(\fg)$,
 \item $E_iF_j - F_jE_i = \delta_{ij}
 \frac{\tilde{K}_i-\tilde{K}_{-i}}{q^{d_i}-q^{-d_i}}$, where
 $\tilde{K}_{\pm i}=K_{\pm d_i \al_i}$,
 \item For all $i\neq j$ \[\sum_{a+b=-c_{ij}+1}(-1)^{a} E_i^{(a)}E_jE_i^{(b)} = 0
 \qquad {\rm and} \qquad
 \sum_{a+b=-c_{ij} +1}(-1)^{a} F_i^{(a)}F_jF_i^{(b)} = 0 .\]
\end{enumerate} \end{center}

This is a Hopf algebra with coproduct on Chevalley generators given
by $$\Delta(E_i)=E_i\otimes 1
+\tilde K_i\otimes E_i\hspace{1cm}\Delta(F_i)=F_i\otimes \tilde K_{-i}
+1 \otimes F_i$$
and antipode on these generators defined by
$S(E_i)=-\tilde{K}_{-i}E_i,S(F_i)=-F_i\tilde{K}_{i}.$

We should note that this choice of coproduct coincides with that of
Lusztig \cite{Lusbook}, but is opposite to the choice in some of our
other references, such as \cite{CP,STtwist}.  In particular, we should
not use the formula for the R-matrix given in these references, but
that arising from Lusztig's quasi-R-matrix.  There is a unique element
$\Theta\in \widehat{U_q^-(\fg)\otimes U_q^+(\fg)}$ such that
$\Delta(u)\Theta=\Theta\bar \Delta(u)$, where $$\bar \Delta(E_i)=E_i\otimes 1
+\tilde K_{-i}\otimes E_i\hspace{1cm}\bar\Delta(F_i)=F_i\otimes \tilde K_{i}
+1 \otimes F_i.$$ If we let $A$ be the operator which acts on weight
vectors by $A(v\otimes w)= q^{\langle \wt(v),\wt(w)\rangle}v\otimes
w$, then as noted by Tingley \cite[2.10]{TingRqR}, $R=A\Theta^{-1}$ is a
universal R-matrix for the coproduct $\Delta$ (which Tingley denotes
$\Delta^{op}$).  This is the opposite of the $R$-matrix of \cite{CP}
(for example).

We let $U_q^\Z(\fg)$ denote the Lusztig (divided powers) integral form
generated over $\Z[\qD,q^{\nicefrac{-1}D}]$ by
$\frac{E_i^n}{[n]_q!},\frac{F_i^n}{[n]_q!}$ for all integers $n$
of this quantum group.  The integral form of the representation of highest weight $\la$ over this quantum group will be denoted by $V_\la^\Z$, and $V_{\bla}^\Z=V_{\la_1}^\Z\otimes_{\Z[\qD,q^{\nicefrac{-1}D}]} \cdots \otimes_{\Z[\qD,q^{\nicefrac{-1}D}]} V_{\la_\ell}^\Z$.  We let $V_\bla=V_\bla^\Z\otimes_{\Z[\qD,q^{\nicefrac{-1}D}]}\Z((\qD))$ be the tensor product with the ring of integer valued Laurent series in $\qD$; this is the completion of $V^\Z_\bla$ in the $q$-adic topology.

We let $\alg^\bla$ be the algebra of red and black strands defined in \cite[\S 2]{WebCTP} and let $\cata^\bla=\alg^\bla-\modu$ be the category of graded finite
  dimensional representations of $\alg^\bla$ graded by $\nicefrac{1}D\Z$.  This is a minor conventional difference with \cite{WebCTP}, where $\Z$-graded modules were used, but this is such a minor difference we felt it did not merit a notational change.  

We let  $\cat^\bla=\Dbe(\cata^\bla)$ be the derived category of complexes of projective objects in $\cata^\bla$ which are 0 in homological degree $j$ and internal degree $i$ if $i+j\ll 0$ or $j\gg 0$ (here we take the convention that the differential increases homological degree).  This notation agrees with that of \cite [\S 2.12]{BGS96}. 

As before, the ring $\Z[\qD,q^{\nicefrac{-1}{D}}]$ acts on $K_0(\alg^\bla)$ by $q^A[M]=[M(A)]$ for any $A\in \frac{1}{D}\Z$.  We note that $K_0(\cat^\bla)$ is the completion of $K_0(\alg^\bla)$ in the $q$-adic topology; thus corresponding to the isomorphism of \cite[Theorem 3.6]{WebCTP}, we also have an isomorphism $K_0(\cat^\bla)\cong V_\bla$ as $\Z((\qD))$-modules.

We will freely use other notation from the companion paper \cite{WebCTP}, but as a courtesy to the reader, we include a list of the most important such notations:
\begin{center}
\begin{tabular}{ll|ll}
$\alg^\bla$ & algebra corresponding to $\bla$&$P_\Bi^\kappa$ & projective module for $(\Bi,\kappa)$\\
$\cata^\bla$ & abelian category of $\alg^\bla$-modules & $S_\Bi^\kappa$ & standard module for $(\Bi,\kappa)$\\
$\cat^\bla$ & derived category of $\alg^\bla$-modules&$\tU$ & the 2-quantum group\\
$\fF_i$ & induction functor & $\mathbb{S}^{\bla}$ & standardization functor $\cat^{\la_1;\dots;\la_\ell}\to\cat^\bla$\\
$\fE_i$ & restriction functor & $\langle -,-\rangle$ & Euler form on $K_0(\alg^\bla)\cong V_\bla$\\
&& $\langle -,-\rangle_1$ & Euler form specialized at $q=1$
\end{tabular}
\end{center}

\subsection*{Acknowledgments} We would like to thank Catharina Stroppel for extremely helpful commentary and pointing out more than one error; and Aaron Lauda, Jim Humphreys, Joel Kamnitzer, Ben Elias, Mikhail Khovanov,  Scott Carter, Eitan Chatav and Kevin Walker for thoughtful conversations and useful feedback.

\section{Braiding functors}
\label{sec:braid-rigid}

\renc{\thethm}{\arabic{section}.\arabic{thm}}

\setcounter{equation}{0}

\subsection{Braiding}
\label{sec:braiding}

Recall that the category of integrable $U_q(\fg)$ modules (of type I)
is a {\bf braided category}; that is, for every pair of representations
$V, W$, there is a natural isomorphism $\sigma_{V,W}\colon V\otimes
W\to W\otimes V$ satisfying various commutative diagrams (see, for
example, \cite[5.2B]{CP}, where the name ``quasi-tensor category'' is
used instead).  This braiding is described in terms of an $R$-matrix
$R\in \widehat{U(\fg)\otimes U(\fg)}$, where we complete the tensor
square with respect to the kernels of finite dimensional
representations, as usual.

As we mentioned earlier, we were left at times with
difficult decisions in terms of reconciling the different conventions
which have appeared in previous work.  One which we seem to be forced
into is to use the opposite $R$-matrix from that usually chosen (for
example in \cite{CP}), which would usually be denoted $R^{21}$. Thus,
we must be quite careful about matching formulas with references such
as \cite{CP}.

Our first task is to describe the braiding in terms of an explicit bimodule
$\bra_{\si}$ attached to each braid. Let us describe the
bimodule $\bra_{\si_k}$ attached to a single positive crossing of the
$k$th and $k+1$st strands.  

Like the algebra $\alg^\bla$, the bimodule $\bra_{\si_k}$ is spanned by
pictures.  In fact, it is spanned by pictures which are identical to
those used in the definition of $\alg^\bla$, except that we must have a
single crossing between the $k$th and $k+1$st red strands.  These
pictures are acted upon on the left by $\alg^\bla$ and on the right by
$\alg^{\si_k\cdot\bla}$ in the obvious way.  More generally, we can view the sum of these over all $\bla$ as a bimodule over the universal algebra $\alg=\oplus_\bla \alg^\bla$.  The module $\bra_{\si_k}$  is homogeneous, where a diagram is assigned a grading as in \cite[\S 2.1]{WebCTP}, but with the red crossing given degree $-\langle\la_k,\la_{k+1}\rangle$.

\begin{figure}[ht]
  \begin{equation*}
    \begin{tikzpicture}
      [very thick,scale=1.4]
      \usetikzlibrary{decorations.pathreplacing} \draw[wei] (-3.5,-1)
      -- +(0,2) node[at start,below]{$\la_1$} node[at
      end,above]{$\la_1$}; \draw (-2.75,-1) -- +(0,2); \node at
      (-2,0){$\cdots$}; \node at (2,0){$\cdots$}; \draw[wei] (3.5,-1)
      -- +(0,2) node[at start,below]{$\la_\ell$} node[at
      end,above]{$\la_\ell$}; \draw (2.75,-1) -- +(0,2); \draw[wei]
      (1.2,-1) -- (-1,1) node[at start,below]{$\la_{k+1}$} node[at
      end,above]{$\la_{k+1}$} node[pos=.55,fill=white,circle]{}; \draw[wei]
      (-1.2,-1) -- (1,1) node[at start,below]{$\la_{k}$} node[at
      end,above]{$\la_{k}$}; \draw (.5,-1) to[in=-90, out=90]
      (1.5,1); \draw (.8,-1) to[in=-90, out=90] (1.8,1); \draw
      (-.5,-1) to[in=-90, out=100] (-1.5,1); \draw (-.8,-1) to[in=-90,
      out=100] (-1.8,1); \node at (0,.7){$\cdots$}; \node at
      (0,-.7){$\cdots$};
    \end{tikzpicture}
  \end{equation*}
  \caption{An example of an element of $\bra_{\si_k}$.}
\end{figure}

As before, we need to mod out by relations:
\begin{itemize}
\item We impose all local relations from \cite[\S 2]{WebCTP}, including
  planar isotopy.
\item Furthermore, we have to add the relations (along with their mirror images)
  
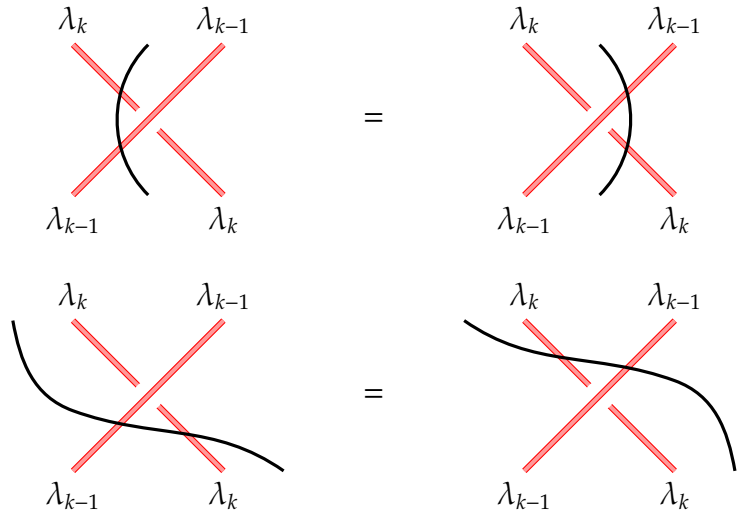
\begin{equation*}
    \begin{tikzpicture}
      [very thick,scale=1] \usetikzlibrary{decorations.pathreplacing}
      \draw[wei] (1,-1) -- (-1,1) node[at start,below]{$\la_k$}
      node[at end,above]{$\la_k$} node[midway,fill=white,circle]{};
      \draw[wei] (-1,-1) -- (1,1) node[at start,below]{$\la_{k-1}$}
      node[at end,above]{$\la_{k-1}$}; \draw (0,-1)
      to[out=135,in=-135] (0,1); \node at (3,0){=}; \draw[wei] (7,-1)
      -- (5,1) node[at start,below]{$\la_k$} node[at
      end,above]{$\la_k$} node[midway,fill=white,circle]{}; \draw[wei] (5,-1)
      -- (7,1) node[at start,below]{$\la_{k-1}$} node[at
      end,above]{$\la_{k-1}$}; \draw (6,-1) to[out=45,in=-45] (6,1);
    \end{tikzpicture}
  \end{equation*}
  \begin{equation*}
    \begin{tikzpicture}
      [very thick,scale=1] \usetikzlibrary{decorations.pathreplacing}
      \draw[wei] (1,-1) -- (-1,1) node[at start,below]{$\la_k$}
      node[at end,above]{$\la_k$} node[midway,fill=white,circle]{};
      \draw[wei] (-1,-1) -- (1,1) node[at start,below]{$\la_{k-1}$}
      node[at end,above]{$\la_{k-1}$}; \draw (1.8,-1)
      to[out=145,in=-20] (-1,-.2) to[out=160,in=-80] (-1.8,1); \node
      at (3,0){=}; \draw[wei] (7,-1) -- (5,1) node[at
      start,below]{$\la_k$} node[at end,above]{$\la_k$}
      node[midway,fill=white,circle]{}; \draw[wei] (5,-1) -- (7,1) node[at
      start,below]{$\la_{k-1}$} node[at end,above]{$\la_{k-1}$}; \draw
      (7.8,-1) to[out=100,in=-20] (7,.2) to[out=160,in=-35] (4.2,1);
    \end{tikzpicture}
  \end{equation*}
\end{itemize}
Following our convention in \cite{WebCTP}, we use $\tilde{\bra}_{\si_k}$ to denote the corresponding $\tilde \alg^\bla-\tilde {\alg}^{\si_k\cdot \bla}$ bimodule where the relation that any diagram with a violating strand is 0 is {\it not} imposed.

Recall
that for any permutation $w$, there is a unique positive braid $\si_w$ which
induces that permutation on the ends of the strands of the same length
of the permutation, constructed by a picking a reduced expression $w=s_{i_1}\cdots s_{i_m}$, and taking the product $\si_w=\si_{i_1}\cdots \si_{i_m}$.  We call this the permutation's {\bf minimal
  lift}.  
Fixing a reduced expression for each $w$, we let
$\tilde\bra_\si=\tilde\bra_{\si_{i_1}}\otimes_{\tilde \alg} \cdots
\otimes_{\tilde \alg} \tilde \bra_{\si_{i_m}}.$

\begin{lemma}\label{lem:independent}
  For $\si$ a minimal lift of $w$, the module $\tilde\bra_\si$ is independent of reduced expression.
\end{lemma}
\begin{proof}
  The module $\tilde\bra_\si$ can be described diagrammatically
  exactly as $\tilde\bra_{\si_k}$ is, except that the red strands
  should cross according to the reduced expression of $w$.  Thus, we
  must show that this module is independent of reduced expression.  

When two distant reflections switch order, this can be achieved via an
isotopy that just changes the $y$-coordinates of their crossings. This
is obviously a canonical isomorphism.

If we have an expression of the form $\cdots s_{i}s_{i+1}s_{i}\cdots$,
then the red strands involved form a triangle, and we can use the
relations to remove all black strands from that triangle.  Applying
the braid relation here simply collapses this red triangle, and
creates a new one also void of strands.  This induces the desired isomorphism.
\end{proof}


Recall from \cite[\S 2.3]{WebCTP} that for any reduced word in $S_{n+\ell}$ which permutes the red strands according to $\si$, we obtain an element $\psi_w\in \tilde{\bra}_{\si_k}$. Fix a choice of reduced word $\Bw$ for each such permutation.
\begin{prop}\label{minimal-basis}
If $\si$ is a minimal lift, then the bimodule $\tilde{\bra}_{\si}$ has a basis given by diagrams $\psi_\Bw$ times an arbitrary monomial in the dots on black strands.  These elements span $\bra_\si=\bra_{\si_{i_1}}\otimes_{\alg} \cdots \otimes_{\alg} \bra_{\si_{i_m}}$.
\end{prop}
\begin{proof}
The proof is essentially identical to that of \cite[Theorem 2.4]{WebCTP}; the argument that these elements span is literally the same.  

Linear independence is slightly more complex.  We note that we have a natural map $\tilde\bra_{\si}\otimes_{\tilde E}\tilde\bra_{\si'}\to \tilde\bra_{\si\si'}$ given by stacking which makes the sum over all positive braids $\si$ into a ring. This ring has a polynomial representation, just like that defined by $\tilde{\alg}_\bla$ in the proof of \cite[Theorem 2.4]{WebCTP}.  This shows that the map $R\to \tilde\bra_\si$ given by horizontal composition is injective (since the image acts by Khovanov and Lauda's polynomial representation). We can reduce to this case by taking any other relation, and composing at the top and bottom with elements pulling all strands to the right.  Thus, a non-trivial relation between our claimed basis vectors would give a nontrivial relation between Khovanov and Lauda's basis for $R$, which is thus a contradiction.
\end{proof}

\begin{defn}
  Let $\mathbb{B}_{\si_k}$ be the functor
  $-\overset{L}\otimes\bra_{\si_k}:D^-(\cata^\bla)\to D^-(\cata^{\si_k\cdot\bla})$.
\end{defn}
Here, $D^-(\cata^\bla)$ refers to the bounded above derived category of $\cata^\bla$;  {\it a priori}, the functor $\mathbb{B}_{\si_k}$ does not obviously preserve the subcategory $\cat^\bla\subset D^-(\cata^\bla)$. In order to show this, and certain other important properties of this functor, we require some technical results.

\begin{prop}\label{bra-commute}
The functors $\mathbb{B}_{\si_k}$ commute with all 1-morphisms in $\tU$.
\end{prop}
\begin{proof}
Of course, we only need to check this for $\fF_i$ and $\fE_i$.  In both cases, there is an obvious map $u\circ \mathbb{B}_{\si_k}\to \mathbb{B}_{\si_k}\circ u$, which is an isomorphism on the $\tilde{}\,$-level, by the basis given in Proposition \ref{minimal-basis}.  The preimage of any element with a violating strand under this map also has a violating strand, so it gives the desired isomorphism.
\end{proof}

\begin{prop}\label{sta-braid}
   \begin{math}\displaystyle
    \mathbb{B}_j\left(\mathbb{S}^{\bla}(P_{\dots;\Bi_j;\emptyset;\dots})\right)
    \cong \mathbb{S}^{\bla}(P_{\dots;\emptyset;\Bi_j;\dots})\Big(\big\langle\la_j-\bal(j),\la_{j+1}\big\rangle\Big)
  \end{math}
\end{prop}
\begin{proof}
We can reduce to the case where the crossing is of the only two
strands.  In this case, $\mathbb{S}^{\bla}(P_{\Bi_j;\emptyset})$ is
projective, so
$\mathbb{B}_j\left(\mathbb{S}^{\bla}(P_{\dots;\Bi_j;\emptyset;\dots})\right)$
is the naive tensor product of these modules.  The isomorphism to
$\mathbb{S}^{\bla}(P_{\dots;\emptyset;\Bi_j;\dots})\Big(\big\langle\la_j-\bal(j),\la_{j+1}\big\rangle\Big)$
is the single diagram shown in Figure \ref{fig:cross-gen}.
  \begin{figure}[ht]
   \centering
    \begin{tikzpicture}
      [very thick,xscale=1.7,yscale=1.4]
      \usetikzlibrary{decorations.pathreplacing}
      \node at  (-1.5,0){$\cdots$}; 
      \node at (2.5,0){$\cdots$}; 
      \draw[wei] (1.2,-1) -- (-1,1) node[at start,below]{$\la_{j+1}$} node[at
      end,above]{$\la_{j+1}$} node[pos=.55,fill=white,circle]{}; 
      \draw[wei]    (-1.2,-1) -- (1,1) node[at start,below]{$\la_{j}$} node[at
      end,above]{$\la_{j}$}; \draw (.5,-1) to[in=-110, out=70]
      (2.4,1); \draw
      (-.4,-1) to[in=-110, out=70] (1.7,1);  \node at
      (.3,-.7){$\cdots$};
\node at      (1.9,.7){$\cdots$};
    \end{tikzpicture}
  \caption{The generator of $\mathbb{B}_j\left(\mathbb{S}^{\bla}(P_{\dots;\Bi_j;\emptyset;\dots})\right)$.}\label{fig:cross-gen}
\end{figure}
\end{proof}

\begin{cor}\label{br-cat}
  The action of $\mathbb{B}_\si$ categorifies the action of the
  braiding.
\end{cor}
\begin{proof}
  By Proposition \ref{bra-commute},
  the induced action on $V_\bla$, which we denote by $R_\si$,
  commutes with the action of $U_q^-(\fg)$.  Thus we need only calculate
  the action of $R_\si$ on a pure tensor of a weight vectors with a
  {\em highest} weight vector $v_h$ in the $j+1$st place, since these generate $V_\bla$ as a $U_q^-(\fg)$ -representation. 

The space of such vectors is spanned by the classes of the form $\mathbb{S}^{\bla}(P_{\dots;\Bi_j;\emptyset;\dots})$.
Thus, Proposition \ref{sta-braid} implies that
  \begin{equation*}
R_\si(v_1\otimes\cdots\otimes v_j\otimes v_h\otimes\cdots\otimes v_\ell)=q^{\langle\wt(v_j),\la_{j+1}\rangle}v_1\otimes\cdots\otimes v_h\otimes v_j\otimes\cdots\otimes v_\ell
  \end{equation*}
which is exactly what the braiding does to vectors of this form as we
noted in the Notation section.
Since vectors of this form generate the representation, there is a unique endomorphism with this behavior, and $R_\si$ is the braiding.
\end{proof}

\begin{lemma}\label{pro-sta}
  If $\si=\si_{i_1}\cdots \si_{i_m}$ is a positive braid, then the functor  $\mathbb{B}_\si=\mathbb{B}_{\si_{i_1}}\cdots \mathbb{B}_{\si_{i_m}}$
  is independent of the choice of word in the generators (up to canonical
  isomorphism).

  If $\si$ is a minimal lift of a
  permutation, then for any projective $P^\kappa_\Bi$, the module $\mathbb{B}_\si(P^\kappa_\Bi)$ has a standard
  filtration and $\mathbb{B}_\si(S^\kappa_\Bi)$ is a module (that is,
  $\Tor_{\alg^\bla}^{>0}(S^\kappa_\Bi,\bra_\si)=0$).
\end{lemma}
\begin{proof}
Note that the independence of choice of word for all positive braids is equivalent to that for minimal lifts, since the braid relations only involve minimal lifts.

  Thus, we need only prove the statements of the theorem for a minimal
  lift $\si$. We will prove these simultaneously by induction on the length of $\si$.  This
  induction is slightly subtle, so rather than attempt each step in
  one go, we break the theorem into 3 statements, and induct around a
  triangle. Consider the three statements (for each positive integer
  $n$):
  \begin{enumerate}
  \item[$p_n:$] For all $\si$ with $\ell(\si)=n$, $\mathbb{B}_\si$
    sends projectives to modules.
  \item[$f_n:$] For all $\si$ with $\ell(\si)=n$, $\mathbb{B}_\si$
    sends projectives to objects with standard filtrations, and is independent of reduced expression.
  \item[$s_n:$] For all $\si$ with $\ell(\si)=n$, $\mathbb{B}_\si$
    sends standards to modules.
  \end{enumerate}
  Our induction proceeds by showing
  \begin{equation*}
    \cdots  \to  p_n \to f_n \rightarrow s_n \rightarrow p_{n+1} \rightarrow \cdots  
  \end{equation*}

  These are all obviously true for $\si=1$, so this covers the base of
  our induction.

  $f_n\rightarrow s_n$: Consider $\Tor^i(S^\kappa_{\Bi},\dot
  S^{\kappa'}_{\Bi'})$.  By symmetry, we may assume that $(\kappa,\Bi)\nleq (\kappa',\Bi')$ in which case $S^\kappa_\Bi$ has a projective resolution where all higher terms are killed by tensor product with $\dot S^{\kappa'}_{\Bi'}$, since they are projective covers of simples which do not appear as composition factors in $S^{\kappa'}_{\Bi'}$.   Thus, we have $\Tor^i(S^\kappa_{\Bi},\dot
  S^{\kappa'}_{\Bi'})=0$.

  Let $\bar\si$ be a reduced positive braid for the inverse of $\si$.
  Then if we let $\dot{\bra_\si}$ be $\bra_\si$ with the left and
  right actions reversed by the dot-anti-automorphism, then
  $\dot{\bra_\si}\cong\bra_{\bar\si}$.

  By $f_n$, the bimodule $\bra_{\bar\si}$ has a standard filtration as
  a right module, so $\bra_{\si}$ has a standard filtration as a left
  module.  Thus, we have $\Tor^i(S^\kappa_{\Bi},\bra^{\la'}_{\la})$
  for $i>0$ and the same holds for any module with a standard
  filtration.

  $s_n\rightarrow p_{n+1}$: We can write
  $\mathbb{B}_\si=\mathbb{B}_{\si'}\mathbb{B}_{\si''}$ where
  $\si',\si''$ are of length $<n+1$. Thus, by assumption,
  $\mathbb{B}_{\si''}$ sends projectives to standard filtered modules,
  and $\mathbb{B}_{\si'}$ sends standards to modules. The result
  follows.

  $p_n\rightarrow f_n$: Since $\mathbb{B}_\si$ sends projectives to
  modules, the bimodule $\bra_\si$ is the naive tensor product of
  those corresponding to individual crossings.  

  Now, we construct the standard filtration on $D=\mathbb{B}_\si
  P^{\kappa}_{\Bi}$. 
Let $\Phi$ be the
  parameter set of the standard filtration on the projective as defined in \cite[\S 3.4]{WebCTP}.   We compose each of these permutations with the permutation of the blocks of black strands between two consecutive red strands according to the action of $\si$ on the red strands at the left of each block.  As before, we can place a partial
  order on these by considering the preorder on the labeling of the tops
  of the strands, and then within each labeling using the Bruhat
  order.  The element $y_\phi$ which we attach to $\phi\in \Phi$ is
  again the diagram which permutes the red and black strands according
  to a reduced word of the permutation.

  We construct a filtration $D_{\leq \phi},D_{<\phi}$
  out of these elements and partial order; while the element $y_\phi$ involves a choice of reduced word, this filtration is independent of it.  Multiplication by $y_\phi$ gives a surjection $d:S^{\kappa_\phi}_{\Bi_\phi}\twoheadrightarrow D_{\leq \phi}/D_{<\phi}$, which we aim to show is an isomorphism.

Since $\mathbb{B}_{\si}$ categorifies the braiding, when $q$ is specialized to 1, it categorifies the permutation map $V_\bla\to V_{\si\cdot \bla}$, and is thus an isometry for $\langle-,-\rangle_1$.  In particular, $$\dim \mathbb{B}_\si=\langle[\alg^{\si\cdot \bla}],\si\cdot [\alg^\bla]\rangle_1=\sum_{\phi\in\Phi} \langle[\alg^{\si\cdot \bla}],[S^{\kappa_\phi}_{\Bi_\phi}]\rangle_1=\sum_{\phi\in\Phi}\dim S^{\kappa_\phi}_{\Bi_\phi}$$
which shows that all the maps $S^{\kappa_\phi}_{\Bi_\phi}\twoheadrightarrow D_{\leq \phi}/D_{<\phi}$ must be isomorphisms.
\end{proof}

\begin{lemma}
  The functor $\mathbb{B}_\si$ sends $\cat^\bla$ to $\cat^{\si\cdot\bla}$.
\end{lemma}
\begin{proof}
   From Lemma \ref{pro-sta}, we find that $\bra_{\si_i}$ considered as a left module (which is the same as $\dot \bra_{\si_i}$) has a finite length free resolution. So any projective module $M$ is sent to a finite length complex; since there are only finitely many indecomposable projectives, the amount which this can decrease the lowest degree is bounded below.  Thus, a complex of projectives in $C^{\uparrow}(\cata^\bla)$ is sent to another collection of projectives in $C^{\uparrow}(\cata^{\si\cdot \bla})$.
\end{proof}

Let $\tau$ be a positive lift of the longest element.  This is
essentially a half twist, but with the blackboard framing, not the one
with ribbon half-twists as well.

Recall that a module $M$ over a standardly stratified algebra is called {\bf tilting} if $M$  has a standard filtration, and $M^\star$ has a filtration by standardizations (which is weaker than a filtration by standards, since those are standardizations of projectives).

\begin{thm}\label{tilting}
The modules $\mathbb{B}_\tau P_{\Bi}^\kappa$ are tilting, and every indecomposable tilting module is a summand of these tiltings.
\end{thm}
\begin{proof}
We show first that $\mathbb{B}_\tau P_{\Bi}^\kappa$ is self-dual.  The pairing that achieves this duality is a simple variant on that described in \cite[\S 1.3]{WebCTP}, where as before, we form a closed diagram and evaluate its constant term. This pairing is pictorially represented in Figure \ref{pairing}.

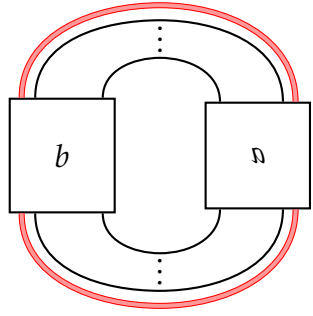
\begin{figure}
\begin{tikzpicture}[thick]
\node (a) [draw,inner sep=.6cm,xscale=-1] at (1.3,0) {$a$};
\node (b) [draw,inner sep=.6cm] at (-1.3,0) {$b$};
\draw[wei] (b.125) to[out=90,in=180] (0,2) to[out=0,in=90] (a.125);
\draw (b.115) to[out=90,in=180] (0,1.8) to[out=0,in=90] (a.115);
\draw (b.55) to[out=90,in=180] (0,1.3) to[out=0,in=90] (a.55);
\draw[wei] (b.-125) to[out=-90,in=180] (0,-2) to[out=0,in=-90] (a.-125);
\draw (b.-115) to[out=-90,in=180] (0,-1.8) to[out=0,in=-90] (a.-115);
\draw (b.-55) to[out=-90,in=180] (0,-1.3) to[out=0,in=-90] (a.-55);
\node at (0,1.65){$\vdots$};
\node[scale=-1] at (0,-1.65){$\vdots$};
\end{tikzpicture}
\caption{The pairing}
\label{pairing}
\end{figure}

The non-degeneracy of this pairing follows from that on $P^0_{\Bi}$.
In \cite[Lemma 3.20]{WebCTP}, we have shown that $P^{\kappa}_{\Bi}$ has an embedding into
$P^{\kappa}_{\Bi}$ into $P^{0}_{\Bi}$ consistent with the standard
filtration, given by left multiplication by the element
$\theta_\kappa$.  By the Tor-vanishing of $\bra_\tau$ paired with any
module with a standard filtration, this map induces an inclusion
$\mathbb{B}_\tau P^{\kappa}_{\Bi}\to P^0_{\Bi}$.

By Proposition \ref{minimal-basis}, any
non-zero diagram in $\mathbb{B}_\tau P^{\kappa}_{\Bi}$ can be drawn
with a section in the middle where all black strands are right of all
red strands.
Thus, the map $P^0_{\Bi}\to P^{\kappa}_{\Bi}$ given by multiplication by
$\dot{\theta_\kappa}$ is not surjective, but the induced
map $P^0_{\Bi}\to \mathbb{B}_\tau P^{\kappa}_{\Bi}$ is.

The pairing of Figure \ref{pairing} is that induced by these maps. This shows
immediately that the perpendicular to the image of the inclusion
contains the kernels of the surjection.  Since these have the same
dimension, they coincide and the pairing is non-degenerate.  Thus,
$\mathbb{B}_\tau P^{\kappa}_{\Bi}$ is self-dual.

By Lemma \ref{pro-sta}, $\mathbb{B}_\tau P^\kappa_\Bi$ has a
filtration by standards for any indecomposable projective;  the
element $\tau$ reverses the pre-order on standards, every standard
which appears is below $(\kappa',\Bi')$ is the sequence obtained from
reversing the blocks of
$(\kappa,\Bi)$, so if $(\kappa,\Bi)$ (and thus $(\kappa',\Bi)$)
is stringy, the tilting whose highest composition factor is the head
of $S^{\kappa'}_{\Bi'}$.  Thus, any tilting is a summand of
$\mathbb{B}_\tau$ applied to a projective.
\excise{
and also, by $\mathbb{B}_\tau S_{\Bi}^\kappa$, which we intend to show are dual standardizations.  Given $M\in \cata^{\la_1;\dots;\la_\ell}$, we have the reverse $M^\tau\in \cata^{\la_\ell;\dots;\la_1}$.  For simplicity, in this proof we will write $\mathbb{S}$ for $\mathbb{S}^\bla$ or $\mathbb{S}^{\la_\ell;\dots;\la_1}$.

Consider $\mathbb{B}_\tau \mathbb{S}M^\tau$.  We have a pairing $\mathbb{B}_\tau \mathbb{S}M^\tau\otimes_{\K} \mathbb{S}(M^\star)\to\K$ defined so that the action of any element of $\alg^\bla$ is self-adjoint under the $\dot{}$-involution, and the normalization $\langle mx_\tau,m'\rangle=\langle m^\tau ,m'\rangle_{M}$ where $x_\tau$ is the element which permutes the groups of black strands along with the red strands to their left (i.e. the unique diagram for any longest word of $\tau$ where each black strand crosses every black and red strand, except the blacks in its group and the red immediately to its left).  If $x$ is a diagram with any other pattern of red/black crossing, $\langle mx_\tau,m'\rangle=0$.  This is well-defined since no diagram we have given a non-zero value lies below any with a 0 value in Bruhat order, and any violating strand in $\bra_\tau$ can be pulled upward to become a standardly violating strand in $\mathbb{S}M^\tau$.

Choosing $v\in\mathbb{B}_\tau \mathbb{S}M^\tau$, we can rewrite the Bruhat longest term in an expression of $v$ as the product of a non-zero element $g$ of $M^tau$ and then an element $x$ of $\bra_\tau$ which never crosses black strands which start in the same group (as usual, if not, we can write $v$ in terms of Bruhat shorter diagrams).  In this diagram in $\bra_\tau$, we can see that every strand must pass to the far right, and move back to the left, but not far enough to cross the red strand that began to its left.  Let $g'\in M^\star$ be such that $\langle g^\tau,g'\rangle\neq 0$.  Let $x'$ be a diagram which completes one of the Bruhat maximal terms of $x$ to $x_\tau$, that is, which creates all black/black crossings between strands not in the same group, and all red/black crossings with red strands the black strand didn't start immediately to the right of that $x$ missed.  Then since all other terms in $xx'$ other than $x_tau$ contribute 0 to the pairing we have $$\langle gx,g'x'\rangle=\langle gx_\tau,g'\rangle=\langle g^\tau ,g'\rangle_{M}\neq 0$$ and the pairing is has no left kernel.  An essentially symmetric argument shows it has no right kernel.  Thus $\mathbb{B}_\tau \mathbb{S}M^\tau\cong \mathbb{S}(M^\star)^\star$, and so is a costandarization.  In particular $\mathbb{B}_\tau P_{\Bi}^\kappa$ is tilting.

For each indecomposable summand of a standard, there is at most one indecomposable tilting with a injection from this costandard, and all tiltings are of this form. Since $S_{\Bi^\tau}^{\kappa^\tau}$ appears as a submodule of  $\mathbb{B}_\tau P_{\Bi}^\kappa$, this shows that all indecomposable tiltings occurs as summands in these.}
\end{proof}

\begin{thm}\label{braid-act}
  The functor $\mathbb{B}_\si$ is an equivalence.
\end{thm}
\begin{proof}
  We will first show $\mathbb{B}_\tau$ is a derived equivalence.
  The higher Ext's between tilting modules always vanish so we always have that $\Ext^{>0}(\mathbb{B}_\tau
  P_{\Bi}^{\kappa},\mathbb{B}_\tau P_{\Bi'}^{\kappa'})=0$; thus we need only show that induced map between endomorphisms of these modules is an isomorphism.

It follows from Corollary \ref{br-cat} that 
\begin{equation*}
  \dim\Hom (\mathbb{B}_\tau P_{\Bi}^{\kappa},\mathbb{B}_\tau
  P_{\Bi'}^{\kappa} )=\langle[\mathbb{B}_\tau P_{\Bi}^{\kappa}],[\mathbb{B}_\tau
  P_{\Bi'}^{\kappa}] \rangle_1= \langle [P_{\Bi}^{\kappa}], [P_{\Bi'}^{\kappa'}]\rangle_1=\dim\Hom ( P_{\Bi}^{\kappa}, P_{\Bi'}^{\kappa'}).
\end{equation*}
The functor $\mathbb{B}_\tau$ induces a map $$\Hom(P_{\Bi}^{\kappa},
P_{\Bi'}^{\kappa})\longrightarrow\Hom(\mathbb{B}_\tau
P_{\Bi}^{\kappa},\mathbb{B}_\tau P_{\Bi'}^{\kappa}).$$ This is
injective, since no element of the image kills the element which pulls all
black strands to the right of all red strands below all crossings, by \cite[Lemma 3.20]{WebCTP}.
Thus, it is surjective by the dimension calculation above.  

It follows
that $\mathbb{B}_\tau$ is an equivalence.  Since it factors through
any $\mathbb{B}_{\si_k}$ on the left and right, the functor $\mathbb{B}_{\si_k}$
is an equivalence as well.
\end{proof}

\begin{lemma}
  The functors $\mathbb{B}_\sigma$ induce a strong action of the braid
  group on the categories $\bigoplus_{w\in S_\ell}\cat^{w\cdot \bla}$.
\end{lemma}
\begin{proof}
  By forthcoming work of Elias and Williamson \cite{ElCox},  it suffices to show that
  we have isomorphisms lifting the braid relations which satisfy the
  Zamolodchikov tetrahedral equations.  This will hold since we have
  defined a canonical functor not just for braid generators, but for
  all positive lifts of permutations.

By Lemma \ref{pro-sta}, the composition $\mathbb{B}_{\si_i}\circ
\mathbb{B}_{\si_{i+1}}\circ  \mathbb{B}_{\si_{i}}$ is the derived
tensor product with $\bra_{\si_i}\otimes \bra_{\si_{i+1}}\otimes
\bra_{\si_i}\cong \bra_{\si_i\si_{i+1}\si_i}$.  By Lemma
\ref{lem:independent}, we have a canonical isomorphism of this functor
with $\mathbb{B}_{\si_{i+1}}\circ
\mathbb{B}_{\si_i}\circ  \mathbb{B}_{\si_{i+1}}$.

Given any reduced expression for the longest permutation of 4
consecutive strands, we can apply these isomorphisms to go around the
loop of the Zamolodchikov tetrahedral equation, collapsing empty red
triangles in the desired sequence.  Since can use the relations to
pull all black strands out of all the polygons created by the red
strands in the permutation of 4 strands, going around this loop sends
a diagram to itself.  Thus, the braid group action we have defined is
strong.
\end{proof}

Recall that the {\bf Ringel dual} of a standardly stratified category is the category of modules over the endomorphism ring of a tilting generator, that is, the opposite category to the heart of the $t$-structure in which the tiltings are projective.  

\begin{cor}
The Ringel dual of $\cata^\bla$ is equivalent to $\cata^{\tau\cdot \bla}$.
\end{cor}

If $C_i$ and $C'_i$ are semi-orthogonal decompositions indexed by $i\in[1,n]$ then $C_i'$ is the {\bf mutation} of $C_i$ by a permutation $\si$ if the category generated by $C_i$ for $i\leq j$ is the same as that generated by $C_{\si(i)}'$ for $i\leq j$.

\begin{prop}\label{pro:mutate}
  For any braid $\si$, $\mathbb{B}_\si$ sends the 
  semi-orthogonal decomposition of \cite[Proposition 3.21]{WebCTP} to its mutation by $\si$.
\end{prop}
\begin{proof}
First, note that we need only show this for $\si_k$.  Of course, an equivalence sends one semi-orthogonal decomposition to another.  Thus, the only point that remains to show is that $\mathbb{B}_{\si_k}(S_\bal)$ for $\bal\leq \boldsymbol{\beta}$ generates the same subcategory as $S'_\bal$ for $\si_k^{-1}(\bal)\leq \boldsymbol{\beta}$, where $S'_\bal$ denotes the appropriate standard module in $\cata^{\si_k\cdot \bla}$.  This follows from the fact that
  \begin{equation*}
    \mathbb{B}_\si S^\kappa_{\Bi}\equiv\mathbb{B}_\si P^\kappa_{\Bi}\equiv S^{\kappa'}_{\Bi'} \text{ modulo smaller $S^{\eta}_{\Bi}$}
  \end{equation*}
  where $\kappa'$ and $\Bi'$ are arrived at by moving the $i$th red
  strand and all black strands between that and the $(i+1)$-st
  rightward to the immediate left of the $(i+2)$-nd.
\end{proof}

\subsection{Serre functors}
\label{sec:serre}

It is a well-supported principle (see, for example, 
Beilinson, Bezrukavnikov and \Mirkovic \cite{BBM04} or Mazorchuk and
Stroppel \cite{MS}) that for any suitable braid group action on a
category, the Serre functor will be given by the full twist.  Here the
same is true, up to grading shift.  Let $\mathfrak{R}=\mathbb{B}_\tau^2$ be the functor
given by a full positive twist of the red strands.  Let $\mathfrak{S}'$ be the functor sends $M\in \cat^\bla_{\al}$ to  $M\big(-\langle\al,\al\rangle+\sum_{i=1}^\ell \langle\la_i,\la_i\rangle\big)$.  Let $\cat_{\mathsf{per}}^\bla$ be the full subcategory of $\cat^\bla$ given by bounded perfect complexes, that is, objects which have finite projective dimension.  We note that in general, this subcategory does not contain many of the important objects in $\cat^\bla$;  for example, it will contain all simple modules if and only if all $\bla$ are minuscule.

\begin{prop}\label{serre}
The right Serre functor of $\cat^\bla_{\mathsf{per}}$ is given by $\mathfrak{S}=\mathfrak{R} \mathfrak{S}'$.
\end{prop}
\begin{proof}
  First consider the action of $\mathfrak{S}$ on
  projective-injectives: the twists of red strands are irrelevant to
  black strands that begin to the right of all of them, so
  $$\mathfrak{R}\cong \operatorname{Id}\big(\lllr-\sum_{i=1}^\ell \langle\la_i,\la_i\rangle\big)$$ as functors on the
  projective-injective category.  We let $I^\kappa_{\Bi}$ be the injective hull of the cosocle of $P^\kappa_{\Bi}$. Since $I_\Bi^0\cong
  P^0_\Bi(\lllr-\langle \al,\al\rangle),$ on this subcategory
  $\mathfrak{S}P^0_\Bi=P^0_\Bi(\lllr-\langle \al,\al\rangle)\cong
  I_\Bi^0$ and so $\mathfrak{S}$ is the graded Serre functor.

  Since they both have costandard and standard filtrations and the same
  class in the Grothendieck group, we have that
  $\mathbb{B}_\tau^{-1}I^\kappa_\Bi$ and $\mathbb{B}_\tau
  P^{\kappa}_{\Bi}$ are the same self-dual tilting module (ignoring
  grading for the moment).  Thus, $\mathfrak{R}P^\kappa_\Bi\cong
  I^\kappa_\Bi$ (again, ignoring the grading). In particular,
  $\mathfrak{R}$ sends projectives to injectives, and is an
  equivalence by Theorem \ref{braid-act}.  By \cite[Theorem 3.4]{MS}, the result follows.
\end{proof}

\begin{samepage}
\section{Rigidity structures}
\label{sec:rigid}

\subsection{Coevaluation and evaluation for a pair of representations}
\label{sec:co-evaluation}

Now, we must consider the cups and caps in our theory.  The most basic case of this is $\bla=(\la,\la^*)$, where we use
$\la^*=-w_0\la$ to denote the highest weight of the dual
representation to $V_\la$.  It is important to note that $V_\la\cong V_{\la^*}^*$, but this isomorphism is not canonical.  

\end{samepage}

In fact, the representation $K_0(\alg^\la)$ comes with more structure, since it is an integral form $V_\bla^\Z$.  In particular, it comes with a distinguished highest weight vector $v_h$, the class of the unique simple in $\cata^\la_\la$ which is 1-dimensional and concentrated in degree 0.  Thus, in order to fix the isomorphism above, we need only fix a lowest weight vector $v_l$ of $V_{\la^*}$, and take the unique invariant pairing such that $\langle v_h,v_l\rangle=1$.

Our first step is to better understand the lowest weight category $\cata^\la_{w_0\la}$: consider a reduced expression $\mathbf{s}$ in the Weyl group $W$ of $\fg$, and let $s_j$ be the product of the first $j$ reflections in
this word.
\begin{defn}\label{longest-sequence}
  Consider the sequence
  \begin{equation*}
    \Bi^\la_{\mathbf{s}}=(i_1^{(\la^{i_1})},i_2^{\left((s_1\la)^{i_2}\right)},\dots, i_k^{\left((s_{k-1}\la)^{i_{k}}\right)})
  \end{equation*}
  Let $g_i$ be the number of times $i$ appears in
  $\Bi^\la_{\mathbf{s}}$ for any reduced expression for the longest
  element $w_0$.  These numbers can also be defined as the unique
  integers so that $\la-w_0(\la)=\sum_i g_i\al_i$.
\end{defn}

\begin{prop}\label{proj-irr}
The projective $P_{\Bi^{\la}_{\mathbf{s}}}^{0}$ over $\alg^\la$ is irreducible, and only depends on the product $s_k\in W$.
\end{prop}
\begin{proof}
Let us show this induction.  The base case is when
  the expression is length 1, which is the case of $\mathfrak{sl}_2$,
  which was shown by Lauda \cite{LauSL21} (this corresponds to the fact
  that the Grassmannian of $k$-planes in $k$-space is a point).

In general, it is clear from 1-dimensionality of extremal weight spaces that the category $\cata^\la_{s_j\la}$ has a unique indecomposable projective and a unique simple, so we need only show that
$\Hom(P_{\Bi^{\la}_{\mathbf{s}}}^{0},P_{\Bi^{\la}_{\mathbf{s}}}^{0})=1$.  Thus, we need
only consider diagrams beginning and ending with our preferred
idempotent.  We claim that such diagrams can be written as a sum of
diagrams where no lines of different colors cross.  This 
reduces our proposition to the $\mathfrak{sl}_2$ case.

Now consider an arbitrary diagram, and consider the left-most block of
strands of a single color whose members cross strands of other colors.
If no strands start in this block at the bottom and end up in a
different block at the top, then we can simply ``pull straight'' and
have a diagram where the first ``bad block'' is further right.

If a strand does leave this block traveling upward, it must be matched
by one which leaves it traveling downward, and the strands must cross.
Using RIII moves, one can move this crossing left (with correction
terms that have fewer such strands, since the correction terms smooth
crossings), so that all differently colored strands pass to its left.
But then at this crossing, we have reordered the strands so that we
get $\Bi_\la^{\mathbf{s}'}$ for some truncation of our word, and then
a repetition of the last element.  This is a composition of induction
functors corresponding to an empty weight space, so is 0.  Thus, by
induction, we are done.
\end{proof}
Fix an expression $\mathbf{s}_0$ for the longest element $w_0$ and consider this construction for $\Bi^\la=\Bi^\la_{\mathbf{s}_0}$.  We fix $v_l=[P^0_{\Bi^{\la^*}}]$, and use this to fix an isomorphism $V_{\la^*}\cong V_\la^*$ which we use freely throughout the rest of the paper.

We can now consider $P_{\Bi^\la}^0$ standardized in two different
ways, obtaining two standard modules: $S^{(0,2\rcl)}_{\Bi_\la}=P^{(0,2\rcl)}_{\Bi_\la}$ and
$S^{0}_{\Bi_\la}$. Proposition \ref{proj-irr} shows that the first has
simple cosocle and the second is itself simple. We denote the cosocles
of these representations by $\Lco$ and $\ocL$.

Recall that the {\bf coevaluation} $\Z((q))\to V_{\la,\la^*}$ is the map sending $1$ to the canonical element of the pairing we have fixed, and {\bf evaluation} is the map induced by the pairing $V_{\la^*,\la}\to \Z((q))$.

\begin{defn}
Let 
$$\mathbb{K}^{\la,\la^*}_{\emptyset}\colon \Dbe(\mathsf{gVect})\to \cat^{\la,\la^*}\text{
be the functor }\RHom_{\K}(\dot{L}_\la,-)(2\llrr)[-2\rcl]$$ \centerline{and} $$\mathbb{E}_{\la^*,\la}^{\emptyset}\colon \cat^{\la^*,\la}\to \Dbe(\mathsf{gVect})\text{ be the functor }\overset{L}\otimes_{\alg^\bla} \dot{L}_{\la^*}$$
\end{defn}
These functors preserve the appropriate categories since by \cite[Theorem 3.16]{WebCTP}, $\Lco$ has a projective resolution in $\Dbe(\cata^\bla)$.  
\begin{prop}\label{coev}
The functor $\mathbb{K}^{\la,\la^*}_{\emptyset}$ categorifies the coevaluation, and $\mathbb{E}_{\la^*,\la}^{\emptyset}$ the evaluation.
\end{prop}
\begin{proof}
Since $\Lco$ is self-dual, we must first check that $[\Lco]$ is invariant.
  Of course, the invariants are the space of vectors of weight 0 such that $\{v|E_iv=0\}$ for any $i$.  Since $P^0_{\Bi^\la}$ has no positive degree endomorphisms, any diagram in which a strand passes over the second red strand is in a proper submodule of $P^{(0,2\rcl)}_{\Bi_\la}$, and so $\fE_i\Lco=0$ for all $i$.  Thus $[\Lco]$ is invariant.
  In fact, $\Lco$ is the only such representation, since the
  $-\la^*$-weight space of $V_\la$ is 1 dimensional.

Now, we need just check the normalization is correct.  Of course, $[\Lco]$'s projection to $(V_\la)_{low}\otimes (V_{\la^*})_{high}$ is \[[ P^{(0,2\rcl)}_{\Bi_\la}]=[P^0_{\Bi^{\la}}]\otimes [P^0_\emptyset]=F_{\Bi^\la}v_h\otimes v_{h^*}.\]  Thus, by invariance, the projection to $(V_\la)_{high}\otimes (V_{\la^*})_{low}$ is $$v_h\otimes S(F_{\Bi^\la})v_{h^*}=(-1)^{2\rcl}q^{-2\llrr}v_h\otimes v_l.$$

On the other hand, one can easily check that $-\overset{L}\otimes_{\alg^\bla}L_{\la^*}$ kills all modules of the form $\fF_iM$, so it gives an invariant map, whose normalization we, again, just need to check on one element.  For example, $P^{(0,2\rcl)}_{\Bi_{\la^*}}\otimes L_{\la^*}\cong\K$, so we get 1 on $v_l\otimes v_h$, which is the correct normalization for the evaluation.
\end{proof}

We represent these functors as leftward oriented cups as is done for the coevaluation and evaluation in the usual diagrammatic approach to quantum groups, as shown in Figure \ref{coev-ev}.

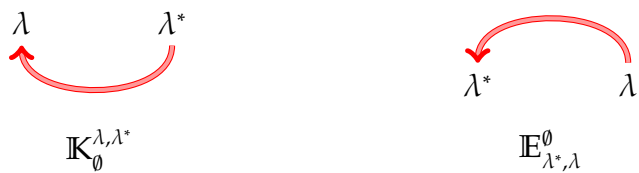
\begin{figure}
\begin{tikzpicture}
\node [label=below:{$\mathbb{K}^{\la,\la^*}_{\emptyset}$}] at (-3,0){
\begin{tikzpicture}
\draw[wei,<-] (-1,0) to[out=-90,in=-90] node[at start, above]{$\la$}
node [at end, above]{$\la^*$} (1,0);
\end{tikzpicture}
};
\node [label=below:{$\mathbb{E}_{\la^*,\la}^{\emptyset}$}] at (3,0){
\begin{tikzpicture}
\draw[wei,<-] (-1,0) to[out=90,in=90] node[at start, below]{$\la^*$}
node [at end, below]{$\la$} (1,0) ;
\end{tikzpicture}
};
\end{tikzpicture}
\caption{Pictures for the coevaluation and evaluation maps.}
\label{coev-ev}
\end{figure}

In order to analyze the structure of $\Lco$ and $\ocL$, we must
understand some projective resolutions of standards.  This can be done with surprising precision in the case where $\ell=2$.

Define a map $\kappa_j:[1,2]\to [0,n]$ by 
$\kappa_j(2)=j$ and $\kappa_j(1)=0.$ 
Given a subset $T\subset [j+1,n]$, we
let $\Bi_{T}$ be the sequence given by $i_1,\dots ,i_j$ followed by
$T$ in reversed sequence, and then $[j+1,n]\setminus T$ in sequence
and let $\kappa_{T}(2)=j+\#T$.  Let $$\chi_T=\sum_{k\in T}\left\langle
  \al_{i_k},-\la_2+\sum_{j< m< k}\al_{i_m}\right\rangle.$$

\begin{prop}\label{pro-res}
  The standard $S_\Bi^{\kappa_j}$ has a projective resolution of the form $$\cdots \longrightarrow \bigoplus_{|T|=n}P_{\Bi_T}^{\kappa_T}(\chi_T)\longrightarrow\cdots \longrightarrow P_\Bi^{\kappa_j} \longrightarrow S_\Bi^{\kappa_j}$$
\end{prop}
\begin{proof}
We induct on $n-j$.  If $j=n$, then $S_\Bi^{\kappa_j}$ is itself projective, so we may take
the trivial resolution.  Let $\Bi'$ be $\Bi$ with its last entry removed, and $\Bi''$ be
$\Bi$ with its last entry moved to the $j+1$st position. As we showed in the proof of \cite[Theorem 3.7]{WebCTP}, we have an exact sequence
$$0\longrightarrow S_{\Bi''}^{\kappa_{j+1}}\Big(\big\langle
\al_{i_n},-\la_2+\sum_{j<
  \ell<n}\al_{i_\ell}\big\rangle\Big)\longrightarrow
\fF_{i_n}S_{\Bi'}^{\kappa_j}\longrightarrow
S_{\Bi}^{\kappa_{j}}\longrightarrow 0. $$ Applying the inductive
hypothesis, we obtain projective resolutions of the left two
factors. Furthermore, we can lift the leftmost map to a map between
projective resolutions.  The cone of this map is the desired
projective resolution of $S_{\Bi}^{\kappa_{j}}$.
\end{proof}

The same principle can be used for any value of $\ell$ to construct an
explicit description of a projective resolution for any standard, but
carefully writing this down is a bit more subtle and difficult than
the $\ell=2$ case, so we will not do so here.  This provides a
resolution of $\ocL$, since it is itself standard.  In particular, it
shows that
\begin{cor}
$\displaystyle{\Ext^i(\ocL,\Lco)=\begin{cases}0& i\neq 2\rcl\\ \K(2\llrr) & i=2\rcl\end{cases}.}$
\end{cor}
\begin{proof}
All of the projectives which appear in the resolution of $\ocL$ has no
maps to $\Lco$ except the last term.  We can break up the grading
shift of this term into the pieces corresponding to simple reflections
in a reduced expression for a longest word of $W$, which are in turn
in canonical bijection is with the set of positive roots $R^+$.  Thus,
we have \[\sum_{i=1}^n\left\langle
  \al_{i_k},-\la^*+\sum_{ m< k}\al_{i_m}\right\rangle=\sum_{\al\in
  R^+}\langle\al, -\la^* \rangle=-2\langle\la^*,\rho\rangle=-2 \llrr\] which is $P_{\Bi_\la}^{(0,\rcl)}(-2\llrr)$.  Thus we have $$\Ext^i(\ocL,\Lco)\cong \Ext^{i-2\rcl}(P_{\Bi_\la}^{(0,2\rcl)}(-2\llrr),\Lco)$$ and the result follows.
\end{proof}

It also shows more indirectly that $\Lco$ has a beautiful, if more
complicated resolution.

\begin{prop}\label{sta-res}
There is a resolution $$ \cdots \longrightarrow M_j\longrightarrow \cdots \longrightarrow M_1\longrightarrow M_0\longrightarrow \Lco\longrightarrow 0$$ of $\Lco$  with the property that \begin{itemize}
\item $M_{2\rcl-j}$ lies in the subcategory generated by $S^{\kappa_j}_{\Bi}$ for all different choices of $\Bi$. In particular, if $j>2\rcl$, then $M_j=0$.
\item $M_{2\rcl}\cong\ocL(-2\llrr)$.
\end{itemize}
\end{prop}
\begin{proof}
  Since we have
  \begin{equation*}
\Ext^i(S^{\kappa_j}_{\Bi},\left(S^{\kappa_k}_{\Bi'}\right)^{\!\star})=0 \quad \text{  if $j\neq k$ or $i>0$,}
\end{equation*}
the first property is equivalent to showing that $$\Ext^m\left(\Lco,
    \left(S^{\kappa_j}_{\Bi}\right)^{\!\star}\right)=0 \text{ if }
  m\neq j.$$ This follows immediately from replacing $S^{\kappa_j}_{\Bi}$ by its
  projective resolution defined in Proposition \ref{pro-res}.  

For the
  second, we must more carefully analyze this $\Ext$ group. By our
  projective resolution, we have
$$\Hom(M_{2\rcl},\left(S^{\kappa_0}_{\Bi_\la}\right)^{\!\star})\cong \Ext^{2\rcl}\left(\Lco,\left(S^{\kappa_0}_{\Bi_\la}\right)^{\!\star}\right)\cong \K(-2\llrr).$$  Thus, we must have $M_{2\rcl}\cong\ocL(-2\llrr)$.
\end{proof}
\begin{cor}
$\displaystyle{\Ext^i(\Lco,\ocL)=\begin{cases}0& i\neq 2\rcl\\ \K(2\llrr) & i=2\rcl\end{cases}.}$
\end{cor}
\begin{cor}\label{LM}
$\displaystyle{\operatorname{Tor}^i(\ocL,\dot{L}_\la)=\begin{cases}0& i\neq 2\rcl\\ \K(-2\llrr) & i=2\rcl\end{cases}.}$
\end{cor}
\excise{\begin{cor}
$\displaystyle{\Ext^{4\rcl}(\Lco,\Lco)\cong \K(-4\llrr)}$.
\end{cor}}

\subsection{Ribbon structure}
\label{sec:ribbon}

This calculation is also important for showing how $\Lco$ behaves under braiding
\begin{prop}\label{Lco-bra}
$\mathbb{B}_{\sigma_1}\Lco\cong L_{\la^*}[-2\rcl](-2\llrr-\lllr)$.
\end{prop}
\begin{proof}
Unless $\Bi$ is a sequence corresponding to weight 0 and $j=\llrr$, we have that  $\bra\overset{L}\otimes \dot P^{\kappa_j}_{\Bi}$ is of the form $\fF_i(\bra\overset{L}\otimes \dot P^{\kappa_j}_{\Bi'})$ for a shorter sequence $\Bi'$. Thus, $\bra\overset{L}\otimes \dot P^{\kappa_j}_{\Bi}$ has a projective resolution in which $P^{\kappa_{\llrr}}_{\Bi}$ never appears, and $$\mathbb{B}\Lco e(\Bi,\kappa_j)\cong \Lco\overset{L}\otimes \bra\overset{L}\otimes \dot P^{\kappa_j}_{\Bi}\cong 0.$$
Thus, we have an isomorphism of  vector spaces $$\mathbb{B}\Lco e(\Bi_\la)\cong \Lco\overset{L}\otimes \bra\overset{L}\otimes\dot P^{\kappa_{\llrr}}_{\Bi_\la}\cong \Lco\overset{L}\otimes \dot M_\la(-\lllr)\cong \K[-2\rcl](-2\llrr-\lllr).$$  As a $\alg^{\la^*,\la}$ representation, $\mathbb{B}\Lco$ must be simple, and thus \begin{equation*}
\mathbb{B}\Lco\cong L_{\la^*}[-2\rcl](-\lllr-2\llrr).\qedhere
\end{equation*}
\end{proof}

Now, in order to define quantum knot invariants, we must also have have quantum trace and cotrace maps, which can only be defined after one has chosen a ribbon structure.  The Hopf algebra $U_q(\fg)$ does not have a unique ribbon structure; in fact topological ribbon elements form a torsor over the characters  $\wela/\rola\to\{\pm 1\}$.  Essentially, this action is by multiplying quantum dimension by the value of the character.  

The standard convention is to choose the ribbon element so that all
quantum dimensions are Laurent polynomials in $q$ with positive
coefficients; however, the calculation above shows that this choice is
not compatible with our categorification! By Proposition
\ref{Lco-bra}, we
have $$\mathbb{B}^2\Lco=\Lco[-4\rcl](-4\llrr-2\lllr).$$ Thus, if we
wish to define a ribbon functor $\mathbb{R}$ to satisfy the
equations $$\mathbb{B}^2\Lco\cong
\mathbb{R}_1^{-2}\Lco=\mathbb{R}_2^{-2}\Lco=\mathbb{R}_1^{-1}\mathbb{R}_{2}^{-1}\Lco,$$
which are necessary for topological invariance (as we depict in Figure
\ref{Lco-inv}).
\begin{defn}
  The {\bf ribbon functor} $\mathbb{R}_i$ is defined by
  \[\mathbb{R}_iM=M[2\rho^\vee(\la_i)](2\langle\la_i,\rho\rangle+\langle\la_i,\la_i
  \rangle).\]
\end{defn}

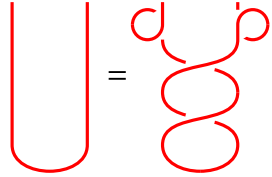
\begin{figure}[ht]
\begin{tikzpicture}[wei, very thick, scale=-1]
\draw (0,0) to node[at end,inner sep=4pt,fill=white,circle]{} (0,.3) to[out=90,in=0] (-.2,.5) to[out=180,in=90] (-.4,.3);
\draw (0,1.2) to[out=-90,in=90]  node[midway,inner sep=4pt,fill=white,circle]{}(1,.5) to node[at end,inner sep=4pt,fill=white,circle]{} (1,.3) to[out=-90,in=180] (1.2,.1) to[out=0,in=-90] (1.4,.3);
\draw (-.4,.3) to[out=-90,in=180] (-.2,.1) to[out=0,in=-90] (0,.3) to[out=90,in=-90] (0,.5) to[out=90,in=-90] (1,1.2) to[out=90,in=-90] node[midway,inner sep=4pt,fill=white,circle]{} (0,1.9) to[out=90,in=180] (.5,2.25);
\draw (.5,2.25) to[out=0,in=90] (1,1.9)[out=-90,in=90] to (0,1.2);
\draw (1.4,.3) to[out=90,in=0] (1.2,.5) to[out=180,in=90] (1,.3) to[out=-90,in=90] (1,0);
\node[black] at (1.6,1){=};
\draw (2,0) to (2,1.9)  to[out=90,in=180] (2.5,2.25) to[out=0,in=90] (3,1.9) to [out=-90,in=90] (3,0) ;
\end{tikzpicture}
\caption{The compatibility of double twist and the ribbon element.}
\label{Lco-inv}
\end{figure}

Taking Grothendieck group, we see that we obtain the ribbon element in
$U_q(\fg)$ uniquely determined by the fact that it acts on the simple
representation of highest weight $\la$ by
$(-1)^{2\rcl}q^{\lllr+2\llrr}$.  This is the inverse  of the ribbon element constructed by
Snyder and Tingley in \cite{STtwist}; we must take inverse because
Snyder and Tingley use the opposite choice of coproduct from ours.  See Theorem 4.6 of that paper
for a proof that this is a ribbon element.  From now on, we will term
this the {\bf ST ribbon element}.  It may seem strange that this
element seems more natural from the perspective of categorification
than the standard ribbon element, but it is perhaps not so surprising;
the ST ribbon element is closely connected to the braid group action
on the quantum group, which also played an important role in Chuang
and Rouquier's early investigations on categorifying $\mathfrak{sl}_2$
in \cite{CR04}.  It is not surprising at all that we are forced into a
choice, since ribbon structures depend on the ambiguity of taking a
square root; while numbers always have 2 or 0 square roots in any
given field (of characteristic $\neq 2$), a functor will often
only have one.

Due to the extra trouble of drawing ribbons, we will draw
all pictures in the blackboard framing.

This different choice of ribbon element will not seriously affect our topological invariants; we simply multiply the invariants from the standard ribbon structure by a sign depending on the framing of our link and the Frobenius-Schur indicator of the label, as we describe precisely in Proposition \ref{schur-indicate}.

\begin{figure}[ht]
\begin{tikzpicture}[wei,very thick]
  \end{tikzpicture}
\begin{tikzpicture}[wei,very thick,scale=1.5,red]

\draw (0,-.3) to[out=90,in=0] node[pos=.57,inner sep=3pt,fill=white,circle]{} (-.5,.5) to[out=180,in=180] (-.5,0) ;
\draw (0,1) to[out=-90,in=135] node[pos=.6,inner sep=3pt,fill=white,circle]{}  (.7,.5) to[in=90,out=-45] (.8,-.3);
\draw[red] (-.5,0) to[out=0,in=-135,<-] (0,.5) to[out=45,in=-90] (.7,1) ; \draw (.7,1) to[out=90,in=90,<-](0,1);
\node[black] at (1.2,.5){=} ;
\draw[red] (1.6,-.3) to[out=90,in=180] (2,1);\draw (2,1) to [out=0,in=90,<-] (2.4,-.3);
  \end{tikzpicture}
\caption{Changing the orientation of a cap}
\label{reverse}
\end{figure}
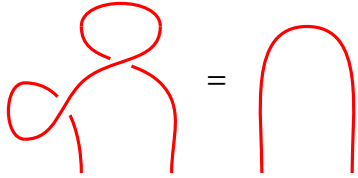

\begin{prop}\label{trace}
The quantum trace and cotrace for the ST ribbon structure are categorified by the functors
$$\mathbb{C}^{\la^*,\la}_{\emptyset}\colon \Dbe(\mathsf{gVect})\to \cat^{\la^*,\la}\text{
given by }\RHom( \dot{L}_{\la^*},-)(2\llrr)[-2\rcl]$$ \centerline{and} $$\mathbb{T}_{\la,\la^*}^{\emptyset}\colon \cat^{\la,\la^*}\to \Dbe(\mathsf{gVect})\text{ given by }-\otimes_{\alg^\bla}\dot{L}_\la.$$
\end{prop}
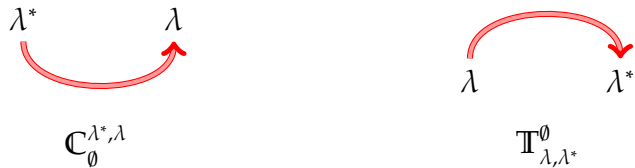
\begin{figure}
\begin{tikzpicture}
\node [label=below:{$\mathbb{C}^{\la^*,\la}_{\emptyset}$}] at (-3,0){
\begin{tikzpicture}
\draw[wei,->] (-1,0) to[out=-90,in=-90] node[at start, above]{$\la^*$}
node [at end, above]{$\la$} (1,0);
\end{tikzpicture}
};
\node [label=below:{$\mathbb{T}_{\la,\la^*}^{\emptyset}$}] at (3,0){
\begin{tikzpicture}
\draw[wei,->] (-1,0) to[out=90,in=90] node[at start, below]{$\la$}
node [at end, below]{$\la^*$} (1,0);
\end{tikzpicture}
};
\end{tikzpicture}
\caption{Pictures for the quantum (co)trace.}
\label{q-tr}
\end{figure}

\begin{proof}
As the picture Figure \ref{reverse} suggests, by definition the quantum trace is given by applying a negative ribbon twist of one strand, and then applying a positive braiding, followed by  the evaluation; that is, it is categorified by \[(\mathbb{B}\mathbb{R}_1-)\otimes \dot{L}_\la\cong -\otimes (\mathbb{B}\mathbb{R}_1\dot{L}_\la)\cong -\otimes \dot{L}_\la.\]  The result thus immediately follows from Proposition \ref{Lco-bra}, and our definition of $\mathbb{R}$.  The same relation between evaluation and quantum trace follows from adjunction.
\end{proof}

\subsection{Coevaluation and quantum trace in general}

More generally, whenever we are presented with a sequence $\bla$ and a dominant weight $\mu$, we wish to have a functor relating the categories $\bla$ and $\bla^+=(\la_1,\dots,\la_{j-1},\mu,\mu^*,\la_{j},\dots,\la_\ell)$.  This will be given by left tensor product with a particular bimodule.

The coevaluation bimodule $\coe^{\bla^+}_{\bla}$ is generated by the diagrams of the form
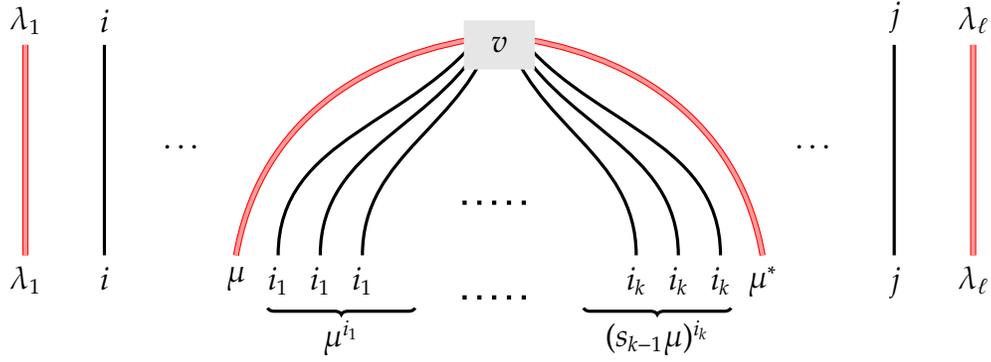
\begin{equation*}
  \begin{tikzpicture}[very thick,xscale=1.4,yscale=-1.4]

\node (v) at (0,-1) [fill=white!80!gray,draw=white!80!gray, thick,rectangle,inner xsep=10pt,inner ysep=6pt, outer sep=-2pt] {$v$};

\begin{pgfonlayer}{background} \begin{scope}[very thick]
\draw[wei] (-4.5,-1) -- +(0,2) node[at start,above]{$\la_1$} node[at end,below]{$\la_1$};
    \draw (-3.75,-1) -- +(0,2) node[at start,above]{$i$} node[at end,below]{$i$};
    \node at (-3,0){$\cdots$};
    \draw[wei] (v.170) to[in=280,out=170] node[at end, below]{$\mu$} (-2.5,1);
\draw[wei] (v.10) to[out=10,in=260] node[at end, below]{$\mu^*$} (2.5,1) ;
  \node at (3,0){$\cdots$};
    \draw[wei] (4.5,-1) -- +(0,2) node[at start,above]{$\la_\ell$} node[at end,below]{$\la_\ell$};
    \draw (3.75,-1) -- +(0,2) node[at start,above]{$j$} node[at end,below]{$j$};
    \draw (v.55) to[in=270,out=55] node[below,at end]{$i_k$} (2.1,1);
    \draw (v.65) to[in=270,out=65] node[below,at end]{$i_k$}(1.7,1);
    \draw (v.75) to[in=270,out=75] node[below,at end]{$i_k$} (1.3,1) ;
    \draw (v.125) to[in=270,out=125]node[below,at end]{$i_1$}  (-2.1,1) ;
    \draw (v.115) to[in=270,out=115] node[below,at end]{$i_1$} (-1.7,1);
    \draw (v.105) to[in=270,out=105] node[below,at end]{$i_1$} (-1.3,1);
    \draw[ultra thick,loosely dotted,-] (-.35,.5) -- (.35,.5);
\draw[ultra thick,loosely dotted,-] (-.35,1.4) -- (.35,1.4);
    \draw[decorate,decoration=brace,-] (-.8,1.5) --
    node[below,midway]{$\mu^{i_1}$} (-2.2,1.5);
    \draw[decorate,decoration=brace,-] (2.2,1.5) --
    node[below,midway]{$(s_{k-1}\mu)^{i_k}$} (.8,1.5) ;\end{scope}
\end{pgfonlayer}
  \end{tikzpicture}
\end{equation*}
where $v$ is an element of $\Lco$ and diagrams only involving the
strands between $\mu$ and $\mu^*$ act in the obvious way.  We impose
the relations: 

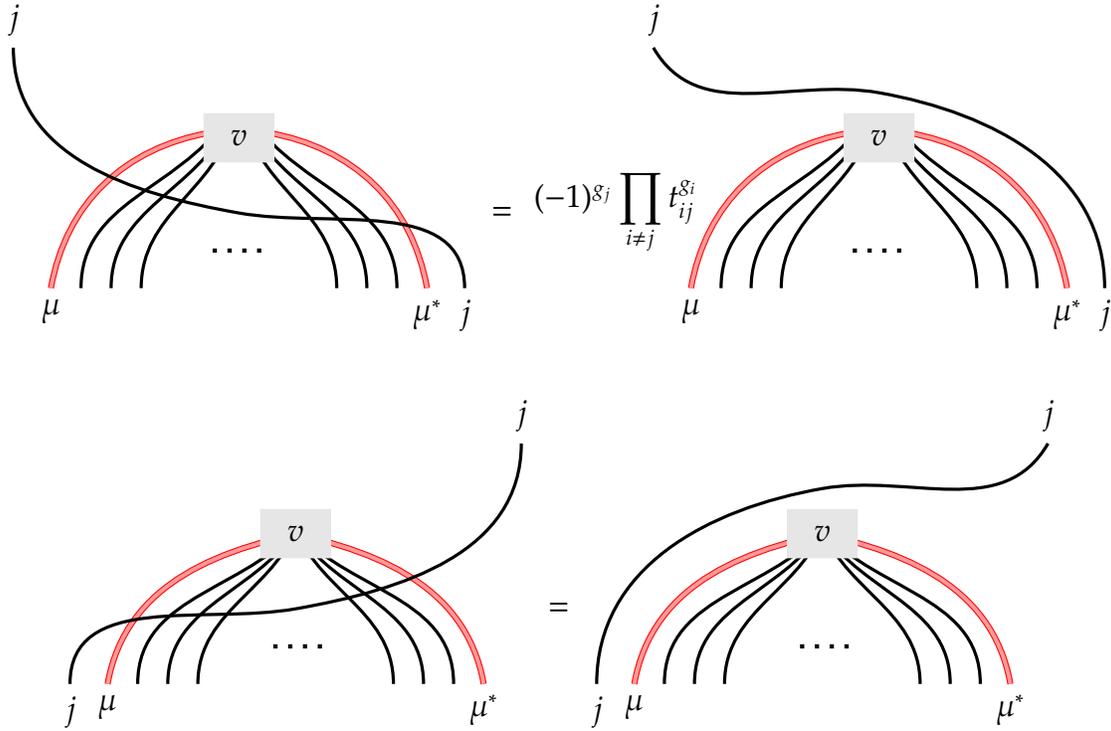
\begin{equation*}
   \begin{tikzpicture}[very thick,yscale=-1]
\node (v) at (0,-1) [fill=white!80!gray,draw=white!80!gray, thick,rectangle,inner xsep=10pt,inner ysep=6pt, outer sep=-2pt] {$v$};
\begin{pgfonlayer}{background} \begin{scope}[very thick]
    \draw[wei] (v.170) to[in=280,out=170] node[at end, below]{$\mu$} (-2.5,1);
\draw[wei] (v.10) to[out=10,in=260] node[at end, below]{$\mu^*$}(2.5,1) ;
    \draw (-3,-2.2) to[in=190,out=90] node[at start,above]{$j$} (0,0)
    to[out=10,in=-90] node[at end,below]{$j$} (3,1)  ;
    \draw (v.55) to[in=270,out=55] node[below,at end]{}(2.1,1);
    \draw (v.65) to[in=270,out=65] node[below,at end]{} (1.7,1);
    \draw (v.75) to[in=270,out=75] node[below,at end]{} (1.3,1);
    \draw (v.125) to[in=270,out=125] node[below,at end]{} (-2.1,1);
    \draw (v.115) to[in=270,out=115] node[below,at end]{} (-1.7,1);
    \draw (v.105) to[in=270,out=105] node[below,at end]{}  (-1.3,1);
    \draw[ultra thick,loosely dotted,-] (-.35,.5) -- (.35,.5);
\end{scope}
\end{pgfonlayer}
\node at (3.5,0){=};
  \end{tikzpicture}
    \begin{tikzpicture}[very thick,yscale=-1]
\node at (-3.5,0) {$(-1)^{g_j}\displaystyle\prod_{i\neq j}t_{ij}^{g_i}$};
\node (v) at (0,-1) [fill=white!80!gray,draw=white!80!gray, thick,rectangle,inner xsep=10pt,inner ysep=6pt, outer sep=-2pt] {$v$};
\begin{pgfonlayer}{background} \begin{scope}[very thick]

    \draw[wei] (v.170) to[in=280,out=170] node[at end, below]{$\mu$} (-2.5,1);
\draw[wei] (v.10) to[out=10,in=260] node[at end, below]{$\mu^*$} (2.5,1);

    \draw  (-3,-2.2) to[in=190,out=60] node[at start,above]{$j$}
    (0,-1.6) to[out=10,in=-90]node[at end,below]{$j$}  (3,1) ;
    \draw (v.55) to[in=270,out=55] (2.1,1) node[below,at end]{};
    \draw (v.65) to[in=270,out=65] (1.7,1) node[below,at end]{};
    \draw (v.75) to[in=270,out=75] (1.3,1) node[below,at end]{} ;
    \draw (v.125) to[in=270,out=125] (-2.1,1) node[below,at end]{};
    \draw (v.115) to[in=270,out=115] (-1.7,1) node[below,at end]{};
    \draw (v.105) to[in=270,out=105] (-1.3,1) node[below,at end]{};
    \draw[ultra thick,loosely dotted,-] (-.35,.5) -- (.35,.5);
\end{scope}
\end{pgfonlayer}
  \end{tikzpicture}
\end{equation*}
\begin{equation*}
   \begin{tikzpicture}[very thick,yscale=-1,xscale=-1]
\node (v) at (0,-1) [fill=white!80!gray,draw=white!80!gray, thick,rectangle,inner xsep=10pt,inner ysep=6pt, outer sep=-2pt] {$v$};
\begin{pgfonlayer}{background} \begin{scope}[very thick]
    \draw[wei] (v.170) to[in=280,out=170] node[at end, below]{$\mu^*$} (-2.5,1);
\draw[wei] (v.10) to[out=10,in=260] node[at end, below]{$\mu$}(2.5,1) ;
    \draw (-3,-2.2) to[in=190,out=90] node[at start,above]{$j$} (0,0)
    to[out=10,in=-90] node[at end,below]{$j$} (3,1)  ;
    \draw (v.55) to[in=270,out=55] node[below,at end]{}(2.1,1);
    \draw (v.65) to[in=270,out=65] node[below,at end]{} (1.7,1);
    \draw (v.75) to[in=270,out=75] node[below,at end]{} (1.3,1);
    \draw (v.125) to[in=270,out=125] node[below,at end]{} (-2.1,1);
    \draw (v.115) to[in=270,out=115] node[below,at end]{} (-1.7,1);
    \draw (v.105) to[in=270,out=105] node[below,at end]{}  (-1.3,1);
    \draw[ultra thick,loosely dotted,-] (-.35,.5) -- (.35,.5);
\end{scope}
\end{pgfonlayer}
\node at (-3.5,0){=};
  \end{tikzpicture}
    \begin{tikzpicture}[very thick,yscale=-1,xscale=-1]
\node (v) at (0,-1) [fill=white!80!gray,draw=white!80!gray, thick,rectangle,inner xsep=10pt,inner ysep=6pt, outer sep=-2pt] {$v$};
\begin{pgfonlayer}{background} \begin{scope}[very thick]

    \draw[wei] (v.170) to[in=280,out=170] node[at end, below]{$\mu^*$} (-2.5,1);
\draw[wei] (v.10) to[out=10,in=260] node[at end, below]{$\mu$} (2.5,1);

    \draw  (-3,-2.2) to[in=190,out=60] node[at start,above]{$j$}
    (0,-1.6) to[out=10,in=-90]node[at end,below]{$j$}  (3,1) ;
    \draw (v.55) to[in=270,out=55] (2.1,1) node[below,at end]{};
    \draw (v.65) to[in=270,out=65] (1.7,1) node[below,at end]{};
    \draw (v.75) to[in=270,out=75] (1.3,1) node[below,at end]{} ;
    \draw (v.125) to[in=270,out=125] (-2.1,1) node[below,at end]{};
    \draw (v.115) to[in=270,out=115] (-1.7,1) node[below,at end]{};
    \draw (v.105) to[in=270,out=105] (-1.3,1) node[below,at end]{};
    \draw[ultra thick,loosely dotted,-] (-.35,.5) -- (.35,.5);
\end{scope}
\end{pgfonlayer}
  \end{tikzpicture}
\end{equation*}
One can think of the relation above as categorifying the equality
$(F_iv)\otimes K=F_i(v\otimes K)$ for any invariant element $K$.

In order to check the coherence of these relations, we need only check
that we can pull a strand which passes over the cup and back either
off the bottom or off using the usual relations, and obtain the same
answer.  That is:
\begin{lemma}
  \begin{equation*}
   \begin{tikzpicture}[very thick,yscale=-1]
\node at (-3.4,0) {$(-1)^{g_j}\displaystyle\prod_{i\neq j}t_{ij}^{g_i}$};\node (v) at (0,-1.3) [fill=white!80!gray,draw=white!80!gray, thick,rectangle,inner xsep=10pt,inner ysep=6pt, outer sep=-2pt] {$v$};
\begin{pgfonlayer}{background} \begin{scope}[very thick]
    \draw[wei] (v.170) to[in=280,out=170] node[at end, below]{$\mu$} (-2.5,1.5);
\draw[wei] (v.10) to[out=10,in=260] node[at end, below]{$\mu^*$}(2.5,1.5) ;
    \draw (3,-1.5) to[in=-100,out=100] node[at start,above]{$j$} node[at end,below]{$j$} (3,1.5)  ;
    \draw (v.55) to[in=270,out=55] node[below,at end]{}(2.1,1.5);
    \draw (v.65) to[in=270,out=65] node[below,at end]{} (1.7,1.5);
    \draw (v.75) to[in=270,out=75] node[below,at end]{} (1.3,1.5);
    \draw (v.125) to[in=270,out=125] node[below,at end]{} (-2.1,1.5);
    \draw (v.115) to[in=270,out=115] node[below,at end]{} (-1.7,1.5);
    \draw (v.105) to[in=270,out=105] node[below,at end]{}  (-1.3,1.5);
    \draw[ultra thick,loosely dotted,-] (-.35,0) -- (.35,0);
\end{scope}
\end{pgfonlayer}
\node at (3.5,0){=};
  \end{tikzpicture}
    \begin{tikzpicture}[very thick,yscale=-1]
\node (v) at (0,-1.3) [fill=white!80!gray,draw=white!80!gray, thick,rectangle,inner xsep=10pt,inner ysep=6pt, outer sep=-2pt] {$v$};
\begin{pgfonlayer}{background} \begin{scope}[very thick]
    \draw[wei] (v.170) to[in=280,out=170] node[at end, below]{$\mu$} (-2.5,1.5);
\draw[wei] (v.10) to[out=10,in=260] node[at end, below]{$\mu^*$}(2.5,1.5) ;
    \draw (3,-1.5) to[in=-20,out=150]   node[at start,above]{$j$}
    (-2.2,-.2) to[out=160,in=-90] (-2.4,0) to[out=90,in=-160] (-2.2,.2)  to [in=-150,out=20] node[at end,below]{$j$} (3,1.5)  ;
    \draw (v.55) to[in=270,out=55] node[below,at end]{}(2.1,1.5);
    \draw (v.65) to[in=270,out=65] node[below,at end]{} (1.7,1.5);
    \draw (v.75) to[in=270,out=75] node[below,at end]{} (1.3,1.5);
    \draw (v.125) to[in=270,out=125] node[below,at end]{} (-2.1,1.5);
    \draw (v.115) to[in=270,out=115] node[below,at end]{} (-1.7,1.5);
    \draw (v.105) to[in=270,out=105] node[below,at end]{}  (-1.3,1.5);
    \draw[ultra thick,loosely dotted,-] (-.35,0) -- (.35,0);
\end{scope}
\end{pgfonlayer}
  \end{tikzpicture}
\end{equation*}
\end{lemma}
\begin{proof}
 This proof will be a bit simpler if we allow ourselves to use both up
 and downward oriented strands (as in the proof of
 \cite[2.11]{WebCTP}), that is using $\eE_i$'s as well as
  $\eF_i$.  We can fix the weight of one of the regions in the
  diagram freely, and the others will be fixed.  We choose to label
  the area outside the cup with weight 0. 

 We begin with the right-hand picture, and
  add a curl.  Push the left side of the curl through the strands. The
  primary term that we arrive at has a curl wrapped over all strands;
  all the correction terms have a strand pulled right out of the cap,
  and thus are 0.  By the relations of $\tU$ (see Figures 2 and 3 of \cite{WebCTP}), this term is multiplied
  by $t_{ij}^{-1}$ each time we cross a strand labeled $i$ for $i\neq
  j$, and by $-1$ when we cross one labeled $j$.  

Thus we obtain the equality:
  \begin{equation}\label{bubble-laid}   
\scalebox{.63}{
\begin{tikzpicture}[very
      thick,yscale=-1,baseline]
\node[scale=1.4] at (2.3,-1.1) {$0$};
\node at (-3.4,0) {$(-1)^{g_j}\displaystyle\prod_{i\neq j}t_{ij}^{g_i}$};
\node (v) at (0,-1.3) [fill=white!80!gray,draw=white!80!gray, thick,rectangle,inner xsep=10pt,inner ysep=6pt, outer sep=-2pt] {$v$};
\begin{pgfonlayer}{background} \begin{scope}[very thick]
    \draw[wei] (v.170) to[in=280,out=170] node[at end, below]{$\mu$} (-2.5,1.5);
\draw[wei] (v.10) to[out=10,in=260] node[at end, below]{$\mu^*$}(2.5,1.5) ;
    \draw (3,-1.5) to[in=-100,out=100] node[at start,above]{$j$} node[at end,below]{$j$} (3,1.5)  ;
    \draw (v.55) to[in=270,out=55] node[below,at end]{}(2.1,1.5);
    \draw (v.65) to[in=270,out=65] node[below,at end]{} (1.7,1.5);
    \draw (v.75) to[in=270,out=75] node[below,at end]{} (1.3,1.5);
    \draw (v.125) to[in=270,out=125] node[below,at end]{} (-2.1,1.5);
    \draw (v.115) to[in=270,out=115] node[below,at end]{} (-1.7,1.5);
    \draw (v.105) to[in=270,out=105] node[below,at end]{}  (-1.3,1.5);
    \draw[ultra thick,loosely dotted,-] (-.35,0) -- (.35,0);
\end{scope}
\end{pgfonlayer}
\node at (3.5,0){=};
  \end{tikzpicture}
\begin{tikzpicture}[very
      thick,yscale=-1,baseline]
\node[scale=1.4] at (2.3,-1.1) {$0$};
\node at (-3.4,0) {$(-1)^{g_j}\displaystyle\prod_{i\neq j}t_{ij}^{g_i}$};
\node (v) at (0,-1.3) [fill=white!80!gray,draw=white!80!gray, thick,rectangle,inner xsep=10pt,inner ysep=6pt, outer sep=-2pt] {$v$};
\begin{pgfonlayer}{background} \begin{scope}[very thick]
    \draw[wei] (v.170) to[in=280,out=170] node[at end, below]{$\mu$} (-2.5,1.5);
\draw[wei] (v.10) to[out=10,in=260] node[at end, below]{$\mu^*$}(2.5,1.5) ;
    \draw [postaction={decorate,decoration={markings,
    mark=at position .5 with {\arrow[scale=1.2]{>}}}},postaction={decorate,decoration={markings,
    mark=at position .06 with {\arrow[scale=1.2]{>}}}},postaction={decorate,decoration={markings,
    mark=at position .94 with {\arrow[scale=1.2]{>}}}}] (3,-1.5) to[in=0,out=100] node[at start,above]{$j$} 
(2.7,.3) to[out=180,in=90] (2.4, 0) to[out=-90,in=180] (2.7,-.3) to [out=0,in=-100]
node[at end,below]{$j$} (3,1.5)  ;
    \draw (v.55) to[in=270,out=55] node[below,at end]{}(2.1,1.5);
    \draw (v.65) to[in=270,out=65] node[below,at end]{} (1.7,1.5);
    \draw (v.75) to[in=270,out=75] node[below,at end]{} (1.3,1.5);
    \draw (v.125) to[in=270,out=125] node[below,at end]{} (-2.1,1.5);
    \draw (v.115) to[in=270,out=115] node[below,at end]{} (-1.7,1.5);
    \draw (v.105) to[in=270,out=105] node[below,at end]{}  (-1.3,1.5);
    \draw[ultra thick,loosely dotted,-] (-.35,0) -- (.35,0);
\end{scope}
\end{pgfonlayer}
\node at (3.5,0){=};
  \end{tikzpicture}
   \begin{tikzpicture}[very thick,yscale=-1,baseline]
\node[scale=1.4] at (2.3,-1.1) {$0$};
\node (v) at (0,-1.3) [fill=white!80!gray,draw=white!80!gray, thick,rectangle,inner xsep=10pt,inner ysep=6pt, outer sep=-2pt] {$v$};
\begin{pgfonlayer}{background} \begin{scope}[very thick]
    \draw[wei] (v.170) to[in=280,out=170] node[at end, below]{$\mu$} (-2.5,1.5);
\draw[wei] (v.10) to[out=10,in=260] node[at end, below]{$\mu^*$}(2.5,1.5) ;
    \draw [postaction={decorate,decoration={markings,
    mark=at position .5 with {\arrow[scale=1.2]{>}}}},postaction={decorate,decoration={markings,
    mark=at position .06 with {\arrow[scale=1.2]{>}}}},postaction={decorate,decoration={markings,
    mark=at position .94 with {\arrow[scale=1.2]{>}}}}] (3,-1.5) to[in=0,out=100] node[at start,above]{$j$} 
(2,.5) to (-2.1,.5) to[out=180,in=90] (-2.6, 0) to[out=-90,in=180] (-2.1, -.5) to (2,-.5) to [out=0,in=-100]
node[at end,below]{$j$} (3,1.5)  ;
    \draw (v.55) to[in=270,out=55] node[below,at end]{}(2.1,1.5);
    \draw (v.65) to[in=270,out=65] node[below,at end]{} (1.7,1.5);
    \draw (v.75) to[in=270,out=75] node[below,at end]{} (1.3,1.5);
    \draw (v.125) to[in=270,out=125] node[below,at end]{} (-2.1,1.5);
    \draw (v.115) to[in=270,out=115] node[below,at end]{} (-1.7,1.5);
    \draw (v.105) to[in=270,out=105] node[below,at end]{}  (-1.3,1.5);
    \draw[ultra thick,loosely dotted,-] (-.35,0) -- (.35,0);
\end{scope}
\end{pgfonlayer}
  \end{tikzpicture}}
\end{equation}
Now, we move the crossing through to the left.
  This also has correction terms coming from the triple point
  relations; these consist of a strands pulled for the rightmost group
  of strands, and then a bubble laid over the remain strands on the
  lefthand side.  We intend to show that all these correction terms
  are 0. 

If the rightmost of these strands is not
  $i$-colored, then we can pull this left, at the cost of adding dots
  on the bubble; if it is $i$-colored, then we can collapse it to a
  crossing.  We gradually push this crossing or the right edge of the
  bubble to the left, eventually we arrive either at a term where a
  strand is pulled out from the cup, which is 0, or where we have a
  positive degree bubble, which is thus also 0, since the diagram
  acting on the simple has negative degree and whose action is
  trivial.
  
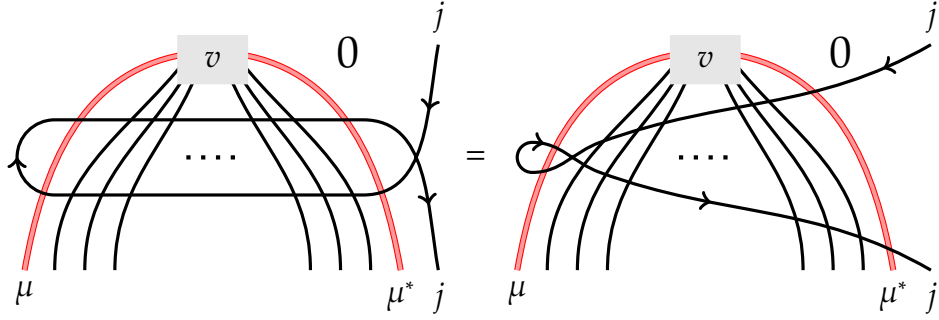
\begin{equation*}   
   \begin{tikzpicture}[very thick,yscale=-1,baseline]
\node[scale=1.4] at (1.8,-1.4) {$0$};
\node (v) at (0,-1.3) [fill=white!80!gray,draw=white!80!gray, thick,rectangle,inner xsep=10pt,inner ysep=6pt, outer sep=-2pt] {$v$};
\begin{pgfonlayer}{background} \begin{scope}[very thick]
    \draw[wei] (v.170) to[in=280,out=170] node[at end, below]{$\mu$} (-2.5,1.5);
\draw[wei] (v.10) to[out=10,in=260] node[at end, below]{$\mu^*$}(2.5,1.5) ;
    \draw [postaction={decorate,decoration={markings,
    mark=at position .5 with {\arrow[scale=1.2]{>}}}},postaction={decorate,decoration={markings,
    mark=at position .06 with {\arrow[scale=1.2]{>}}}},postaction={decorate,decoration={markings,
    mark=at position .94 with {\arrow[scale=1.2]{>}}}}] (3,-1.5) to[in=0,out=100] node[at start,above]{$j$} 
(2,.5) to (-2.1,.5) to[out=180,in=90] (-2.6, 0) to[out=-90,in=180] (-2.1, -.5) to (2,-.5) to [out=0,in=-100]
node[at end,below]{$j$} (3,1.5)  ;
    \draw (v.55) to[in=270,out=55] node[below,at end]{}(2.1,1.5);
    \draw (v.65) to[in=270,out=65] node[below,at end]{} (1.7,1.5);
    \draw (v.75) to[in=270,out=75] node[below,at end]{} (1.3,1.5);
    \draw (v.125) to[in=270,out=125] node[below,at end]{} (-2.1,1.5);
    \draw (v.115) to[in=270,out=115] node[below,at end]{} (-1.7,1.5);
    \draw (v.105) to[in=270,out=105] node[below,at end]{}  (-1.3,1.5);
    \draw[ultra thick,loosely dotted,-] (-.35,0) -- (.35,0);
\end{scope}
\end{pgfonlayer}\node at (3.5,0){=};
  \end{tikzpicture}
   \begin{tikzpicture}[very thick,yscale=-1,baseline]
\node[scale=1.4] at (1.8,-1.4) {$0$};
\node (v) at (0,-1.3) [fill=white!80!gray,draw=white!80!gray, thick,rectangle,inner xsep=10pt,inner ysep=6pt, outer sep=-2pt] {$v$};
\begin{pgfonlayer}{background} \begin{scope}[very thick]
    \draw[wei] (v.170) to[in=280,out=170] node[at end, below]{$\mu$} (-2.5,1.5);
\draw[wei] (v.10) to[out=10,in=260] node[at end, below]{$\mu^*$}(2.5,1.5) ;
    \draw [postaction={decorate,decoration={markings,
    mark=at position .54 with {\arrow[scale=1.2]{>}}}},postaction={decorate,decoration={markings,
    mark=at position .06 with {\arrow[scale=1.2]{>}}}},postaction={decorate,decoration={markings,
    mark=at position .74 with {\arrow[scale=1.2]{>}}}}] (3,-1.5) to[in=-20,out=150] node[at start,above]{$j$} 
(-1.4,-.2) to[out=160,in=0] (-2.3,.2) to[out=180,in=90] (-2.5, 0) to[out=-90,in=180] (-2.3, -.2) to[out=0,in=-160] (-1.4,.2) to [out=20,in=-150]
node[at end,below]{$j$} (3,1.5)  ;
    \draw (v.55) to[in=270,out=55] node[below,at end]{}(2.1,1.5);
    \draw (v.65) to[in=270,out=65] node[below,at end]{} (1.7,1.5);
    \draw (v.75) to[in=270,out=75] node[below,at end]{} (1.3,1.5);
    \draw (v.125) to[in=270,out=125] node[below,at end]{} (-2.1,1.5);
    \draw (v.115) to[in=270,out=115] node[below,at end]{} (-1.7,1.5);
    \draw (v.105) to[in=270,out=105] node[below,at end]{}  (-1.3,1.5);
    \draw[ultra thick,loosely dotted,-] (-.35,0) -- (.35,0);
\end{scope}
\end{pgfonlayer}
  \end{tikzpicture}
\end{equation*}
In order to finish, we apply this relation to the loop at the far left:
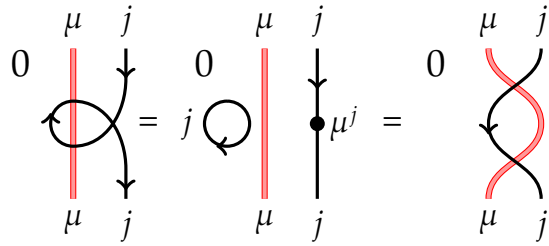
\begin{equation*}
  \begin{tikzpicture}[very thick,baseline]
\node[scale=1.2] at (-.7,.8){$0$};
    \draw[wei] (0,-1) --  node[at end, above]{$\mu$} node[at start, below]{$\mu$} (0,1);
   \draw [postaction={decorate,decoration={markings,
    mark=at position .5 with {\arrow[scale=1.2]{<}}}},postaction={decorate,decoration={markings,
    mark=at position .06 with {\arrow[scale=1.2]{<}}}},postaction={decorate,decoration={markings,
    mark=at position .94 with {\arrow[scale=1.2]{<}}}}] (.7,-1) to node[at start,below]{$j$} 
(.7,-.6) to[out=90,in=0] (0,.3) to[out=180,in=90] (-.3, 0) to[out=-90,in=180] (0, -.3) to[out=0,in=-90] (.7,.6) to 
node[at end,above]{$j$} (.7,1);
\node at (1,0){$=$};
  \end{tikzpicture}
  \begin{tikzpicture}[very thick,baseline]
\node[scale=1.2] at (-.3,.8){$0$};
    \draw[wei] (.5,-1) --  node[at end, above]{$\mu$} node[at start, below]{$\mu$} (.5,1);
   \draw [postaction={decorate,decoration={markings,
    mark=at position .5 with {\arrow[scale=1.2]{<}}}}] (0,.3)
to[out=180,in=90] (-.3, 0) to[out=-90,in=180] node[at start,left]{$j$}
(0, -.3) to[out=0,in=-90] (.3,0) to[out=90,in=0] (0,.3);
\draw [postaction={decorate,decoration={markings,
    mark=at position .8 with {\arrow[scale=1.2]{<}}}}] (1.2,-1) -- node [midway,fill=black, inner sep=2pt, circle,label=right:{$\mu^j$}]{} node[at end,above]{$j$} node[at start,below]{$j$} (1.2,1);
\node at (2.2,0){$=$};
  \end{tikzpicture}
\begin{tikzpicture}[very thick,baseline]
\node[scale=1.2] at (-.2,.8){$0$};
    \draw[wei] (.5,-1) to[out=90,in=-90] node[at start, below]{$\mu$}  
    (1.2,0) to[out=90,in=-90] node[at end, above]{$\mu$} (.5,1);
\draw [postaction={decorate,decoration={markings,
    mark=at position .5 with {\arrow[scale=1.2]{<}}}}] (1.2,-1) to[out=90,in=-90] node[at start,below]{$j$}  (.5,0) to[out=90,in=-90]  node[at end,above]{$j$} (1.2,1);
  \end{tikzpicture}
\end{equation*}

This shows that the RHS of \eqref{bubble-laid} is equal to the RHS of
the desired result, completing the proof.
\end{proof}

 Let $\fF_{\Bi}^{\kappa}$ denote composition of functors where one reads the corresponding idempotent from left to right, and applies $\fF_i$ when passing a black strand labeled $i$, and $\fI_\la$ when passing a red strands labeled $\la$.  This has the useful property that $\fF_{\Bi}^{\kappa}P_\emptyset=P^\kappa_\Bi$.  

We  can write $\bla=\bla'\bla''$ and $\Bi=\Bi'\Bi''$ as the union of the red/black strands that come before and after the point where $\mu,\mu^*$ are inserted, with $\kappa',\kappa''$ be the corresponding $\kappa$-functions.  Then, we can give an alternate definition of this bimodule by the formula.
  $$P^\kappa_{\Bi}\otimes \coe^{\kappa^+}_{\kappa}\cong \fF_{\Bi''}^{\kappa''}\big(\mathbb{S}^{\bla';(\mu,\mu*)}(P^{\kappa}_{\Bi'}\boxtimes L_\mu)\big).$$

\begin{defn}
  The coevaluation functor is
  \begin{equation*}
\mathbb{K}^{\bla^+}_{\bla}=\RHom_{\alg^{\bla}}( \coe^{\bla^+}_\bla, -)(2\llrr)[-2\rcl]\colon\cat^\bla\to \cat^{\bla+}.
\end{equation*}

Similarly, the quantum trace functor is the left adjoint to this given by
  \begin{equation*}
\mathbb{T}^{\bla^+}_{\bla}=-\overset{L}\otimes_{\alg^{\bla^+}}\coe^{\bla^+}_\bla\colon\cat^{\bla^+}\to \cat^{\bla}.
\end{equation*}
The evaluation and quantum cotrace are defined similarly.
\end{defn}
Since $\coe^{\bla^+}_\bla$ is projective as a right module, $\Hom$ with it gives an exact functor.  The quantum trace functor, however, is very far from being exact.
\begin{prop}\label{qt-cat}
  $\mathbb{K}^{\bla^+}_{\bla}$ categorifies the coevaluation and $\mathbb{T}^{\bla^+}_{\bla}$ the quantum trace.
\end{prop}
\begin{proof}
We need only prove the former, since the latter follows by adjunction.  Furthermore, we may reduce to the case where $\mu$ is added at the end
of the sequence, since all other cases are obtained from this by the
action of $\tU$.

In this case, consider
\begin{math}
\mathbb{K}^{\bla^+}_{\bla}(S^\kappa_{\Bi}).
\end{math}
The resulting module is isomorphic to the standardization
\begin{equation*}
  \mathbb{S}^{\bla;\mu,\mu^*}(S^\kappa_{\Bi}\boxtimes L_\mu)(2\llrr)[-2\rcl]
\end{equation*}
since any diagram with a left crossing involving the red lines from $\la_m$'s is trivial since we are considering a standardization and any with a left crossing on the strand labeled $\mu$ is killed since it is positive degree.

This reduces to the case where $\bla=\emptyset$, which we have covered in Propositions \ref{coev} and \ref{trace}.
\end{proof}

The most important property of these functors is that they satisfy
the obvious isotopy; there are two functors
\[S_1=\mathbb{T}^{\bla_1\la;\la^*,\la;\bla_2}_{\bla_1\la\bla_2}\mathbb{K}^{\bla_1;\la,\la^*;\la\bla_2}_{\bla_1\la\bla_2}\qquad S_2=\mathbb{T}^{\bla_1;\la,\la^*;\la\bla_2}_{\bla_1\la\bla_2}\mathbb{K}^{\bla_1\la;\la^*,\la;\bla_2}_{\bla_1\la\bla_2}\]
which come from adding a pair of the representations
are added on the left of an entry $\la$, and removing them on the
right of $\la$ or {\it vice versa}.
\begin{prop}\label{S-move}
  The functors $S_1$ and $S_2$ are isomorphic to the identity functor.
\end{prop}
\begin{proof}
  As in Proposition \ref{qt-cat}, we can easily reduce to the case
  where $\bla_1=\bla_2=\emptyset$.  Furthermore, these functors
  commute with the action of $\tU$, and so it suffices to check this
  equality on $P_\emptyset$.  To prove the result for $S_2$,  we must check that
  \[\mathbb{S}^{\la;\la^*,\la}(P_\emptyset\boxtimes
  L_\la)\overset{L}\otimes_{\alg^\bla} \mathbb{S}^{\la,\la^*;\la}(\dot L_\la\boxtimes
  \dot P_\emptyset)(2\llrr)[-2\rcl]\cong \K\]
Applying the dot involution to switch left/right, the symmetry of
tensor product shows that $S_1$ reduces to the same calculation.

We can use Lemma \ref{sta-res} to expand $L_\la$ into a complex, and
then use the spectral sequence attached to tensoring these complexes.
The $E^2$-page of this spectral sequence has entries \[E^2_{k,m}=\bigoplus_{i+j=m}\operatorname{Tor}^k\Big(\mathbb{S}^{\la;\la^*,\la}(P_\emptyset\boxtimes
  M_i), \mathbb{S}^{\la,\la^*;\la}(\dot M_j\boxtimes
  \dot P_\emptyset)(2\llrr)\Big)\]

By the Tor-vanishing discussed in the proof of \ref{pro-sta}, this
will be 0 unless the two factors lie in the same piece of the
semi-orthogonal decomposition, that is, if $i=0, j=2\rcl$ and $k=0$.  This
term is exactly 
\begin{equation*}
  \mathbb{S}^{\la;\la^*;\la}(P_\emptyset\boxtimes P_{\Bi_\la}\boxtimes P_\emptyset)
  )\overset{L}\otimes_{\alg^\bla} \mathbb{S}^{\la,\la^*;\la}(\dot
  P_\emptyset\boxtimes \dot P_{\Bi_\la}\boxtimes
  \dot P_\emptyset)[-2\rcl]\cong \K[-2\rcl]
\end{equation*}
with the homological shift canceling the fact that $j=2\rcl$.  Thus,
the result follows. \end{proof}

  \begin{figure}
     \centering
 \tikzset{knot/.style={draw=white,double=red,line width=3.5pt, double
     distance=1.2pt}}
 \begin{tikzpicture}[very thick,knot,xscale=1.5]
 \node (a) at (-2.5,0){\begin{tikzpicture}[xscale=.8]
 \draw[knot,postaction={decorate,decoration={markings,
    mark=at position .5 with {\arrow[red,scale=.4]{>}}}}] (-2,1) to[out=-45,in=180] (-1,.1) to[out=0,in=180] (0,.9) to[out=0,in=135] (1,0);
  \end{tikzpicture}};
\node (d)  at (0,0){\begin{tikzpicture}[xscale=.8]
\draw[knot,postaction={decorate,decoration={markings,
    mark=at position .5 with {\arrow[red,scale=.4]{>}}}}] (1,1.5) -- (1,0);
\end{tikzpicture}};
\node (e) at (2.5,0){\begin{tikzpicture}[xscale=-.8]
 \draw[knot,postaction={decorate,decoration={markings,
    mark=at position .5 with {\arrow[red,scale=.4]{>}}}}] (-2,1) to[out=-45,in=180] (-1,.1) to[out=0,in=180] (0,.9) to[out=0,in=135] (1,0);
  \end{tikzpicture}};
\draw[->,black,thick] (a) -- (d);\draw[->,black,thick] (d) -- (e);
 \end{tikzpicture}
     \caption{The ``S-move''}
     \label{Smove}
  \end{figure}
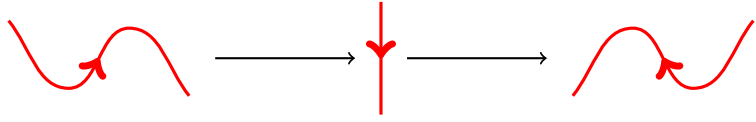

This move is depicted in more usual topological form in
Figure~\ref{Smove}; it is extremely tempting to conclude that this
proposition shows that the functors $\mathbb{K}$ and $\mathbb{T}$ are biadjoint; in
fact, they are not always, though the adjunction on one side is clear
from the definition.  Rather, this is reflecting some sort of
biadjunction between the 2-functors of ``tensor with $\cata^\la$'' and
``tensor with $\cata^{\la^*}$'' on the 2-category of representations
of $\tU$.  While there is not a unified construction of a tensor
product of two $\tU$ categories, one can easily generalize the
definition of $\cata^\bla$ to describe auto-2-functors of $\tU$
representations given by adding one red line; we will discuss this
construction in more detail in forthcoming work \cite{WebCB}.

\section{Knot invariants}
\label{sec:invariants}

\setcounter{equation}{0}

\subsection{Constructing knot and tangle invariants}

Now, we will use the functors from the previous section to construct tangle invariants. Using these as building blocks, we can associate a
functor $\Phi(T)\colon\cat^\bla\to \cat^{\bmu}$ to any diagram of an oriented labeled ribbon tangle  $T$ with the bottom ends given by $\bla=\{\la_1,\dots,\la_\ell\}$ and the
top ends labeled with $\bmu=\{\mu_1,\dots,\mu_m\}$.

As usual, we choose a projection of our tangle such that at any height
(fixed value of the $x$-coordinate) there is at most a single crossing, single
cup or single cap. This allows us to write our tangle as a composition of
these elementary tangles.

For a crossing, we ignore the orientation of the knot, and separate
crossings into positive (right-handed) and negative (left-handed) according to the upward
orientation we have chosen on $\R^2$.
\begin{itemize}
\item To a positive crossing of the $i$ and $i+1$st strands, we associate the braiding functor
  $\mathbb{B}_{\si_i}$.  \item To a negative crossing, we associate its
  adjoint $\mathbb{B}_{\si_i^{-1}}$ (the left and right adjoints are isomorphic,
  since $\mathbb{B}$ is an equivalence).
\end{itemize}
For the cups and caps, it is necessary to consider the orientation, following the pictures of Figures \ref{coev-ev} and \ref{q-tr}.
\begin{itemize}
\item To a clockwise oriented cup, we associate the coevaluation.
\item To a clockwise oriented cap, we associate the quantum trace.
\item  To a counter-clockwise cup, we associate the quantum cotrace.
\item  To a counter-clockwise cap, we associate the evaluation.
\end{itemize}

\begin{prop}\label{decat-all}
 The map induced by $\Phi(T):\cat^\bla\to\cat^\bmu$ on the Grothendieck groups $V_{\bla}\to V_\bmu$ is that assigned to a ribbon tangle by the structure maps of the category of $U_q(\fg)$ with the ST ribbon structure.

In particular, the graded Euler characteristic of the complex $\Phi(T)(\K)$ for a closed link is the quantum knot invariant for the ST ribbon element.
\end{prop}
\begin{proof}
  We need only check this for each elementary tangle, which was done in  Corollary \ref{br-cat}, Section \ref{sec:ribbon} and Proposition \ref{qt-cat}.
\end{proof}

\begin{thm}
The cohomology of $\Phi(T)(\K)$ is finite-dimensional in each homological degree, and each graded degree is a complex with finite dimensional total cohomology.  In particular the bigraded Poincar\'e series 
\[\varphi(T)(q,t)=\sum_{i}(-t)^{-i}\dim_qH^i(\Phi(T)(\K))\]  is a well-defined element of $\Z[\qD,q^{\nicefrac {-1}D}]((t))$.
\end{thm}

\begin{proof}
We note that the category $\cat^\emptyset$ is the category of complexes of graded finite dimensional vector spaces $$\cdots\longleftarrow M^{i+1}\longleftarrow M^i\longleftarrow M^{i-1}\longleftarrow\cdots$$ such that $M^i=0$ for $i\gg 0$ and for some $k$, the vector space $M^i$ is concentrated in degrees above $k-i$.  Thus, $\Phi(T)(\K)$ lies in this category.  In particular, each homological degree and each graded degree of $\Phi(T)(\K)$ is finite-dimensional.  
\end{proof}

The only case where the invariant is known to be finite dimensional is when the representations $\bla$ are {minuscule}; recall that a weight $\mu$ is called {\bf minuscule} if every weight with a non-zero weight space in $V_\mu$ is in the Weyl group orbit of $\mu$. 

\begin{prop}\label{fin-dim}
If all $\la_i$ are minuscule, then the cohomology of $\Phi(T)(\K)$ is finite-dimensional.
\end{prop}
\begin{proof}
If all $\la_i$ are minuscule, then the standard modules form a full exceptional collection.  Any category with a finite full exceptional collection where each element has a finite projective resolution has finite projective dimension.  Thus, in this case, the functor given by $\RHom$ or $\overset{L}\otimes$ with a finite dimensional module preserves being quasi-isomorphic to a finite length complex.
\end{proof}

\subsection{The unknot for $\fg=\mathfrak{sl}_2$}
Unfortunately,  the cohomology of the complex $\Phi(T)(\K)$ is not always finite-dimensional.  This can be seen in examples as simple as the unknot $U$ for $\fg=\mathfrak{sl}_2$ and label $2$.   

In this case, the module $L_2$ with has a standard resolution of the form 
\[0\longrightarrow S^{(0,0)}_{1^2}(-2)\longrightarrow S^{(0,1)}_{1,1}/(y_1+y_2)(-1)\longrightarrow S^{(0,2)}_{1^2}\longrightarrow L_{\bla}\longrightarrow 0.\]

We let $A=\End_{\cat^{2,2}}(S^{(0,1)}_{1,1},S^{(0,1)}_{1,1})\cong \K[y_1,y_2]/(y_1^2,y_2^2)$; the middle piece of the semi-orthogonal decomposition is equivalent to representations of this algebra.

Taking $\otimes$ of this resolution to its dual, we observe that all $\Tor$'s vanish between terms that do not lie in the same piece of the semi-orthogonal decomposition, so \begin{multline*} \Tor^\bullet(\Lco,\Lco)=\Tor^\bullet(S^{(0,2)}_{1^2},(S^{(0,2)}_{1^2})^\star)\\ \oplus \Tor^\bullet(S^{(0,1)}_{1,1}/(y_1+y_2),(S^{(0,1)}_{1,1}/(y_1+y_2))^\star)[2](-2)\oplus \Tor^\bullet(S^{(0,2)}_{1^2},(S^{(0,2)}_{1^2})^\star)[4](-4) \\
\cong \K\oplus\Tor^\bullet_{A}(A/(y_1+y_2)A,A/(y_1+y_2)A)[2](-2)\oplus  \K[4](-4)\\
\end{multline*}
The module $A/(y_1+y_2)A$ has a minimal projective resolution given by \[\cdots\overset{y_1+y_2}\longrightarrow A(-4)\overset{y_1-y_2}\longrightarrow A(-2)\overset{y_1+y_2}\longrightarrow  A\longrightarrow A/(y_1+y_2)A\longrightarrow 0.\]

which after taking $\otimes$ becomes \[\cdots\qquad A/(y_1+y_2)A(-4)\overset{y_1-y_2}\longrightarrow A/(y_1+y_2)A(-2)\qquad  A/(y_1+y_2)A\overset{\sim}\longrightarrow A/(y_1+y_2)A\longrightarrow 0.\]

Thus, we have that \[\Tor^i_{A}(A/(y_1+y_2)A,A/(y_1+y_2)A)\cong \begin{cases}A/(y_1+y_2)A & i=0\\
\K(-2i) & i>0, \text{ odd}\\
\K(-2i-2) & i>0, \text{ even}\end{cases}\]

Thus, we have that 
\begin{prop}
\(\displaystyle \vp(U)=q^{-2}t^2+1+q^2t^{-2}+\frac{q^{-2}-q^{-2}t}{1-t^{2}q^{-4}}\).
\end{prop}

It is easy to see that the Euler characteristic is
$q^{-2}+1+q^2=[3]_q$, the quantum dimension of $V_2$.  As this example
shows, infinite-dimensionality of invariants is extremely typical
behavior, and quite subtle.  This same phenomenon of infinite dimensional vector spaces categorifying integers has also appeared in the work of Frenkel, Sussan and Stroppel \cite{FSS}, and in fact, their work could be translated into the language of this paper using the equivalences of \cite[\S 4]{WebCTP}; it would be quite interesting to work out this correspondence in detail.

\begin{conj}
The invariant $\Phi(L)$ for a link $L$ is only finite-dimensional if all components of $L$ are labeled with minuscule representations. 
\end{conj}

\subsection{Independence of projection}

While Theorem \ref{decat-all} shows the action on the Grothendieck group is independent of the presentation of the tangle, it doesn't establish this for the functor $\Phi(T)$ itself.

\begin{thm}\label{ind}
The functor $\Phi(T)$ does not depend (up to isomorphism) on the projection of $T$.
\end{thm}
\begin{proof}
  We have already proved the ribbon Reidemeister moves in at least one
  position: RI in Proposition \ref{Lco-bra} and RII and RIII as part of Theorem
  \ref{braid-act}, and also the ``S-move'' shown in Figure~\ref{Smove} in Proposition
  \ref{S-move}. There is only one move of importance left for us
  to establish: the pitchfork move, shown in Figure~\ref{pitch-pic}.
  
  Once we have established this move, we can easily show the others
  which are necessary.  The illustrative example of the 
  ``$\chi$-move'' is given in
  Figure \ref{chimove}. The other moves in the list
  of Ohtsuki \cite[Theorem 3.3]{Oht} follow in  the same way.

   \begin{figure}
     \centering
 \tikzset{knot/.style={draw=white,double=red,line width=3.5pt, double
     distance=1.2pt}}
 \begin{tikzpicture}[very thick,knot,xscale=1.3]
 \node (a) at (-7,0){\begin{tikzpicture}[xscale=.8]
 \draw[knot,postaction={decorate,decoration={markings,
    mark=at position .3 with {\arrow[red,scale=.4]{>}}}}] (-2,1)
to[out=-45,in=180] (-1,.1) to[out=0,in=180] (0,.9) to[out=0,in=135]
(1,0);
\draw[knot,postaction={decorate,decoration={markings,
    mark=at position .7 with {\arrow[red,scale=.4]{<}}}}] (-1,1) --(0,0);
  \end{tikzpicture}};
\node (b) at (-3.8,0){\begin{tikzpicture}[xscale=.8]
 \draw[knot,postaction={decorate,decoration={markings,
    mark=at position .3 with {\arrow[red,scale=.4]{>}}}}] (-2,1)
to[out=-45,in=180] (-1,.1) to[out=0,in=180] (0,.9) to[out=0,in=135]
(1,0);
\draw[knot,postaction={decorate,decoration={markings,
    mark=at position .7 with {\arrow[red,scale=.4]{<}}}}] (1,1) --(0,0);
\end{tikzpicture}};
\node (c) at (-.7,0){\begin{tikzpicture}[xscale=.8]
\draw[knot,postaction={decorate,decoration={markings,
    mark=at position .25 with {\arrow[red,scale=.4]{<}}}}] (1,1)
--(-.3,-.3);
\draw[knot,postaction={decorate,decoration={markings,
    mark=at position .75 with {\arrow[red,scale=.4]{>}}}}] (-.3,1) --(1,-.3);
\end{tikzpicture}};
\draw[->,black,thick] (a) -- (b);\draw[->,black,thick] (b) -- (c);
 \end{tikzpicture}
     \caption{The ``$\chi$-move''}
     \label{chimove}
  \end{figure}
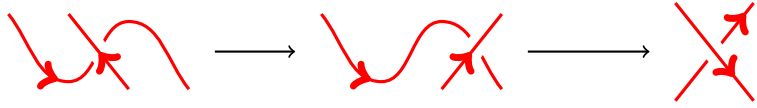

  So, let us turn to the pitchfork.  We may
  assume that the pictured red strands are the only ones.  We must
  prove that this move holds for all reflections and orientations.
  The vertical reflection of the version shown follows from that
  illustrated by adjunction.  We may assume that the cup is clockwise
  oriented, since the counter clockwise move can be derived from that
  one using Reidemeister moves II and III.  The orientation of the
  ``middle tine'' is irrelevant, so we will ignore it.

For the orientation shown in Figure \ref{pitch-fig}, we need only show this move holds for $P_\emptyset$ again, since we again have commutation with Hecke functors.

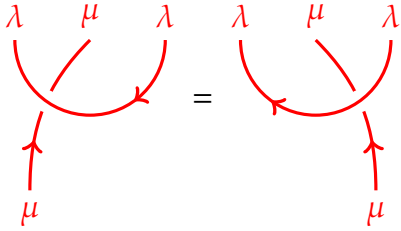
\begin{figure}[ht]
\begin{tikzpicture}[very thick, red, xscale=-1]
\draw[postaction={decorate,decoration={markings,
    mark=at position .3 with {\arrow[red,scale=1.3]{>}}}}] (-.8,-1)
to[out=90,in=-135] node[pos=.55,inner sep=3pt,fill=white,circle]{}
node[at start,below]{$\mu$} node[at end,above]{$\mu$}  (0,1);
\draw[postaction={decorate,decoration={markings,
    mark=at position .7 with {\arrow[red,scale=1.3]{>}}}}] (-1,1)  to[out=-90,in=180]   node[at start,above]{$\la$}(0,0) to[out=0,in=-90]  node[at end,above]{$\la$} (1,1);
\node[black] at (1.5,.2){=};

\draw[postaction={decorate,decoration={markings,
    mark=at position .3 with {\arrow[red,scale=1.3]{>}}}}] (3.8,-1)
to[out=90,in=-45] node[pos=.55,inner sep=3pt,fill=white,circle]{}
node[at start,below]{$\mu$} node[at end,above]{$\mu$} (3,1);
\draw[postaction={decorate,decoration={markings,
    mark=at position .3 with {\arrow[red,scale=1.3]{>}}}}] (2,1) to[out=-90,in=180]  node[at start,above]{$\la$} (3,0) to[out=0,in=-90] node[at end,above]{$\la$}  (4,1);
\end{tikzpicture}
\caption{The ``pitchfork'' move}
\label{pitch-pic}
\end{figure}

We have two functors $\cat^{\la,\la^*}_0\to \cat^{\la,\mu,\la^*}_\mu$ given by
\begin{equation*}
\Pi_1=\mathbb{B}_{\si_1^{-1}}\circ \mathbb{S}^{\mu,\la+\la^*}(P_\emptyset\boxtimes -)\hspace{1in}\Pi_2=\mathbb{B}_{\si_2}\circ \mathbb{S}^{\la+\la^*,\mu}(-\boxtimes P_\emptyset).
\end{equation*}

\begin{lemma}\label{pitch}
The functors $\Pi_1$ and $\Pi_2$ coincide.
\end{lemma}
\begin{proof}
  First, we multiply both sides by $\mathbb{B}_{\si_2}$, so we must show
  that we have isomorphisms of functors $$
  \mathbb{S}^{\mu,\la+\la^*}(P_\emptyset\boxtimes
  -)\cong\mathbb{B}_{\si_1}\circ \mathbb{B}_{\si_2}\circ\mathbb{S}^{\la+\la^*,\mu}(-\boxtimes P_\emptyset).$$ Since
  they generate the category, we need only show this isomorphism can
  be exhibited on the level of projectives.

The isomorphism is given by Figure \ref{pitch-fig}, and is essentially the same as that of Proposition \ref{sta-braid}.  We note
that this element has degree zero because we are assuming that the
roots on the black strands add to $\la+\la^*$.  Any diagram in the module $\mathbb{B}_{\si_1}\mathbb{B}_{\si_2}\mathbb{S}^{\la+\la^*,\mu}(P_{\Bi}^\kappa\boxtimes P_\emptyset)$  can be prefixed by this element, so the map is surjective.  Any element which is sent to 0 by adjoining this diagram is easily seen to be 0, since the standardly violating strand can be slid downward to become a violating strand, so the map is also injective.  \end{proof}

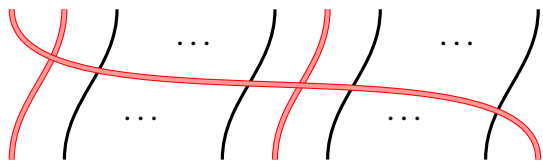
\begin{figure}[ht]
  \begin{tikzpicture}[very thick,xscale=3.5,yscale=-2]
  \draw[wei] (-1,.5) to[out=-90,in=90] (-.8,-.5);
  \draw (-.8,.5) to[out=-90,in=90] (-.6,-.5);
  \node at (-.5,.25) {$\cdots$};  \node at (-.3,-.25) {$\cdots$};
  \draw (-.2,.5) to[out=-90,in=90] (0,-.5);
  \draw[wei] (0,.5) to[out=-90,in=90] (.2,-.5);
  \draw (.2,.5) to[out=-90,in=90] (.4,-.5);
  \node at (.5,.25) {$\cdots$};  \node at (.7,-.25) {$\cdots$};
  \draw (.8,.5) to[out=-90,in=90] (1,-.5);
  \draw[wei] (1,.5) to[out=-90,in=90] (-1,-.5);
  \end{tikzpicture}
  \caption{The isomorphism of Lemma \ref{pitch}}\label{pitch-fig}
\end{figure}

The pitchfork move shown in Figure \ref{pitch-pic} follows from this
lemma, since two sides of the depicted move are $$-\otimes_{T} \Pi_1\Lco(2\llrr)[-2\rcl]\quad\text{ and }\quad-\otimes_T\Pi_2\Lco(2\llrr)[-2\rcl].$$ The only variation remaining to check is the case where
the move is reflected through the page (i.e. with the signs of the
crossings given reversed), but this follows from the lemma as well
since the two sides are \begin{equation*}-\otimes_T(\Pi_1\Lco)^\star(2\llrr)[-2\rcl]\quad\text{ and }\quad-\otimes_T(\Pi_2\Lco)^\star(2\llrr)[-2\rcl].\qedhere\end{equation*}
\end{proof}

Some care must be exercised with the normalization of these
invariants, since as we noted in Section \ref{sec:ribbon}, they are
the Reshetikhin-Turaev invariants for a slightly different ribbon
element from the usual choice.  However, the difference is easily
understood.  Let $L$ be a link drawn in the blackboard framing, and
let $L_i$ be its components, with $L_i$ labeled with $\la_i$.  Recall
that the {\bf writhe} $\wr(K)$ of a oriented ribbon knot is the
linking number of the two edges of the ribbon; this can be calculated
by drawing the link the blackboard framing and taking the difference
between the number of positive and negative crossings.  Here we give a
slight extension of the proposition of Snyder and Tingley relating the
invariants for different framings \cite[Theorem 5.21]{STtwist}:
\begin{prop}\label{schur-indicate}
  The  invariants attached to $L$ by the standard and Snyder-Tingley ribbon elements differ by the scalar $\prod_{i}(-1)^{2\rho^\vee(\la_i)\cdot(\wr(L_i)-1)}$.
\end{prop}
\begin{proof}
The proof is essentially the same as that of  \cite[Theorem 5.21]{STtwist} with a bit more attention paid to the case where the components have different labels.  The proof is an induction on the crossing number of the link.  The formula is correct for any framing of an unlink, which gives the base case of our induction.  

Now note that the ratio between the knot invariants only depends on the number of rightward oriented cups and caps, so both the ratio between the invariants for the usual and ST ribbon structures and the formula given are insensitive to Reidemeister II and III as well as crossing change (which changes the writhe, but by an even number).    Since these operations can be used to reduce any link to an unlink, we are done.
\end{proof}

Since one of the main reasons for interest in these quantum invariants
of knots is their connection to Chern-Simons theory and invariants of
3-manifolds, it is natural to ask:
\begin{ques}
  Can these invariants glue into a categorification of the
  Witten-Reshetikhin-Turaev invariants of 3-manifolds?
\end{ques}
\begin{rem}
The most naive ansatz for categorifying Chern-Simons theory, following
the development of Reshetikhin and Turaev \cite{RT91} would associate
\begin{itemize}
\item a category $\cC(\Sigma)$ to each surface $\Sigma$, and
\item an
object in $\cC(\Sigma)$ to each isomorphism of $\Sigma$ with the boundary of
a 3-manifold
\end{itemize}
such that 
\begin{itemize}
\item the invariants $\EuScript{K}$ we have given are the Ext-spaces
of this object for a knot complement with fixed generating set of
$\cC(T^2)$ labeled by the representations of $\fg$, and 
\item the categorification of the WRT
invariant of a Dehn filling is the Ext space of this object with
another associated to the torus filling. 
\end{itemize}
 While some hints of this
structure appear in the constructions of this paper, it's far from
clear how they will combine.
\end{rem}

\subsection{Functoriality}
One of the most remarkable properties of Khovanov homology is its functoriality with respect to cobordisms between knots \cite{Jac04}.  This property is not only theoretically satisfying but also played an important role in Rasmussen's proof of the unknotting number of torus knots \cite{Ras04}. Thus, we certainly hope to find a similar property for our knot homologies.  While we cannot present a complete picture at the moment, there are promising signs, which we explain in this section.  We must restrict ourselves to the case where the weights $\la_i$ are minuscule, since even the basic results we prove here do not hold in general.  We will assume this hypothesis throughout this subsection.

The weakest form of functoriality is putting a Frobenius structure on
the vector space associated to a circle.  This vector space, as we
recall, is $$A_\la=\Ext^\bullet(\Lco,\Lco)[2\rcl](2\llrr).$$  This algebra is naturally bigraded by the homological and internal gradings.
The algebra structure on it is that induced by the Yoneda product.  Recall that $\mathfrak{S}$ denotes the right Serre functor of $\cat^\bla$, discussed in Section \ref{sec:serre}.
\label{Frob-conj}
\begin{thm}
  For minuscule weights $\bla$, we have a canonical isomorphism $$\mathfrak{S}\Lco\cong
  \Lco(-4\llrr)[-4\rcl].$$ Thus, the functors $\mathbb{K}$ and
  $\mathbb{T}$ are biadjoint up to shift.

  In particular, $\Ext^{4\llrr}(\Lco,\Lco)\cong \Hom(\Lco,\Lco)^*$,
  and the dual of the unit $$\iota^*\colon\Ext^{4\llrr}(\Lco,\Lco)\to
  \K$$ is a symmetric Frobenius trace on $A_\la$ of degree $-4\llrr$
\end{thm}

One should consider this as an analogue of Poincar\'e duality, and thus is a piece of evidence for $A_\la$'s relationship to cohomology rings. 
\begin{proof}
As we noted in the proof of \ref{fin-dim}, $\alg^\bla$ has finite global dimension if the weights $\bla$ are minuscule. The result then follows immediately from Proposition \ref{serre}.
\end{proof}

It would be enough to show that this algebra is commutative to establish the functoriality for flat tangles; we
simply use the usual translation between 1+1 dimensional TQFTs and
commutative Frobenius algebras (for more details, see the book by Kock
\cite{Kock}).  At the moment, not even this very weak form of functoriality is known.

\begin{ques}
  Is there another interpretation of the algebra $A_\la$?  Is it the
  cohomology of a space?
\end{ques}
One natural guess, based on the work of Mirkovi\'c-Vilonen \cite{MV}
and the symplectic duality conjecture of the author and collaborators
\cite{BLPWgco}, is that $A_\la$ is the cohomology
of the corresponding Schubert variety $\overline{\mathrm{Gr}_\la}$ in the
Langlands dual affine Grassmannian.

Another candidate algebra is the multiplication induced on $V_\la$ by
the quantized ``shift of function algebra'' $\EuScript{A}_f$ for a
regular nilpotent element $f$ studied by Feigin, Frenkel, and Rybnikov
\cite{FFR}.

We can use the biadjunction to give a rather simple
prescription for functoriality: for each embedded cobordism in
$I\times S^3$ between knots in $S^3$, we can isotope so that the
height function is a Morse function, and thus decompose the cobordism
into handles.  Furthermore, we can choose this so that the projection
goes through these handle attachments at times separate from the times
it goes through Reidemeister moves.  We construct the functoriality
map by assigning
\begin{itemize}
\item to each Reidemeister move, we associate a fixed isomorphism of
  the associated functors.
\item to the birth of a circle (the attachment of a 2-handle), we associate
  the unit of the adjunction $(\mathbb{K},\mathbb{T})$ or $(\mathbb{C},\mathbb{E})$, depending on the orientation.
\item to the death of a circle (the attachment of a 0-handle), we
  associate the counits of the opposite adjunctions  $(\mathbb{T},\mathbb{K})$ or $(\mathbb{E},\mathbb{C})$ (i.e., the Frobenius trace).
\item to a saddle cobordism (the attachment of a 1-handle), we
  associate (depending on orientation) the unit of the second adjunction above, or the counit of the first.
\end{itemize}

\begin{conj}
  This assignment of a map to a cobordism is independent of the choice
  of Morse function, i.e.\ this makes the knot homology theory
  $\EuScript{K}(-)$ functorial.
\end{conj}

In the case of $\mathfrak{sl}_2$, there is a homology theory which we
believe to coincide with ours, defined by Cooper, Hogancamp and
Krushkal \cite{CoKr, CoHoKr}.  A version of functoriality for this theory has been
given by Hogancamp \cite{Hofunc}, overcoming some of the difficulties
posed by the failure of finite global dimension this case, but still
not giving an answer for every cobordism between knots.

\section{Comparison to other knot homologies}
\label{sec:comparison-functors}

A great number of other knot homologies have appeared on the scene in the last decade, and obviously, we would like to compare them to ours.  While several of these comparisons are out of reach at the moment, in this section we check the one which seems most straightforward based on the similarity of constructions: we describe an isomorphism to the invariants constructed by Mazorchuk-Stroppel and Sussan for the fundamental representations of $\mathfrak{sl}_n$.

To do this, we will use the functor $\Xi:\cata^\bla\to \tcO^\fp$ constructed in \cite[\S 4]{WebCTP} (as before, we will freely use notation from this preceding paper).  Here we use $\bla$ to construct a Young pyramid $\pi$ whose column lengths are the indices of the fundamental weights appearing in the expansion of $\la_j$, and let $\fp$ be a parabolic subalgebra of $\mathfrak{gl}_{N}$ which precisely preserves a flag of type corresponding to the pyramid $\pi$. Given this data, we let $\tcO^\fp$ be a graded lift of a block of $\fp$-parabolic category $\cO$.  

In order to compare knot homologies, we must compare the functors we have described on our categories $\cat^\bla$ and those on $\tcO^\fp$.  For simplicity, in this section we will assume that $\bla$ is a sequence of fundamental weights.  In this paper, we are only concerned about commuting up to isomorphism of functors; thus when we say a diagram of functors ``commutes'' we mean that the functors for any two paths between the same points are isomorphic.

First, let us consider the braiding functors.  Associated to each
permutation of $N$ letters, we have a derived twisting functor
$T_w\colon \Dbe(\tcO)\to \Dbe(\tcO)$ (see \cite{AS} for more details and the definition).  

\begin{prop}\label{trans-braid}
When $\bla=(\om_1,\cdots,\om_1)$, then $\fp=\mathfrak{b}$ and we have a commutative diagram
\begin{equation*}
    \begin{tikzpicture}[yscale=1.1,xscale=1.9,very thick]
        \node (a) at (1,1) {$\Dbe(\tcO_n)$};
        \node (b) at (-1,1) {$\Dbe(\tcO_n)$};
        \node (c) at (1,-1) {$\cat^{\bla}$};
        \node (d) at (-1,-1) {$\cat^\bla$};
        \draw[->] (b) -- (a) node[above,midway]{$T_{v}$};
        \draw[->] (d) -- (c) node[below,midway]{$\mathbb{B}_v$};
        \draw[->] (c) --(a) node[right,midway]{$\Xi$};
         \draw[->] (d) --(b) node[left,midway]{$\Xi$};
    \end{tikzpicture}
\end{equation*}
\end{prop}
\begin{proof}
  We note that functors $T_{v}$ commutes with translation functors by \cite[Lemma 2.1(5)]{AS}.  The same holds for  $\Xi\circ \mathbb{B}_v\circ
  \Xi$ by \cite[Proposition 4.9]{WebCTP} and Proposition \ref{bra-commute}.
  
So as usual, we need only compute their behavior on parabolic Verma
modules on the level of objects in order to check isomorphisms of
functors. Furthermore,  by Proposition
  \ref{pro:mutate}, $\mathbb{B}_v$ 
  sends the exceptional collection of standard objects to its mutation
  by using $v$ to reorder the root function $\bal$ given
  by the sum of the roots that appear between the red lines.  Furthermore $T_v$ sends the exceptional
  collection of parabolic Verma modules to its mutation by the
  change of order associated to
  the action of $v$ on tableaux.  By \cite[4.10]{WebCTP}, these
  changes of partial order are intertwined by the correspondence
  between standard modules and parabolic Verma modules given by
  $\Xi$. Thus the
  mutations also match under $\Xi$,
  so the diagram commutes.  
\end{proof}

Finally, we turn to describing the functors associated to cups and
caps.  If $\pi$ has a column of height $n$ in the $k$th position,
then any block of category $\tcO_n^\fp$ is equivalent to the block of
category $\tcO^{\fp'}_n$ associated to $\pi'$, the diagram $\pi$ with
that column of height $n$ removed. The content of the
tableaux in the new block is that of the original block with the multiplicity of each
number in $[1,n]$ reduced by 1.  The effect of this functor on the simples,
projectives and Vermas is simply removing that column of height $n$ (which by column
strictness must be the numbers $[1,n]$ in order).
The functor that realizes this equivalence $\zeta:\tcO^\fp_n\to
\tcO^{\fp'}_n$ is the {\bf Enright-Shelton equivalence}, which is
developed in the form most useful for us in \cite[\S 3.2]{Sussan2007}.

Having already developed the equivalence $\Xi$, this functor is
actually quite easy to describe.  Let $P^\kappa_d$ denote the
module attached to $\kappa$ and $d$ for $\fp'$ as above, and let
$Q^{\kappa_+}_d$ be the module attached in the same way to $\fp$,
where \[\kappa_+(j)=
\begin{cases}
  \kappa(j) & j\leq k\\
  \kappa(j-1) &j>k.
\end{cases}\]
The result \cite[4.7]{WebCTP} gives equivalences of $\cata^\bla$ with the category generated
by $\operatorname{pr}_n(P^\kappa_d)$ and with that generated by
$\operatorname{pr}_n(Q^\kappa_d)$; under these two equivalences,
$\operatorname{pr}_n(P^\kappa_d)$ and $\operatorname{pr}_n(Q^\kappa_d)$
are sent to the same projective.  The functor $\zeta$ is the
composition of the second equivalence with the inverse of the first.

We will also use also have {\bf Zuckerman functors}, which are the derived functors of
sending a module in $\tcO$ to its largest quotient  which is
locally finite for $\fp$.  These are left adjoint to the
forgetful functor $D^b(\tcO^\fp)\to D^b( \tcO)$.

Begin with a pyramid $\pi$, and assume $\pi'$ is obtained from $\pi$
by replacing a pair of consecutive columns whose lengths add up to $n$
(a pair of consecutive dual representations in the sequence $\bla$),
with one of length $n$, and $\pi''$ is obtained by deleting them
altogether.  
\begin{defn}
The {\bf ES-cup functor} $K\colon\tcO^{\pi''}\to \tcO^{\pi}$ 
is the composition of the inverse of the Enright-Shelton equivalence for
$\pi''$ and $\pi'$ with the forgetful functor from $\tcO^{\pi'}$ to
$\tcO^\pi$ (which corresponds to an inclusion of parabolic subgroups).

The {\bf ES-cap functor} $T\colon\tcO^{\pi}\to \tcO^{\pi''}$  is  the composition of the
Zuckerman functor from $\tcO^\pi$ to $\tcO^{\pi'}$  with the
Enright-Shelton functor $\zeta\colon\tcO^{\pi'}\to\tcO^{\pi''}$.
\end{defn}

\begin{prop}\label{trans-cup-cap}
Both squares in the diagram below commute.
\begin{equation*}
    \begin{tikzpicture}[yscale=1.2,xscale=2.1,very thick]
        \node (a) at (1,1) {$\Dbe(\tcO^{\fp}_n)$};
        \node (b) at (-1,1) {$\Dbe(\tcO^{\fp'}_n)$};
        \node (c) at (1,-1) {$\cat^{\bla^+}$};
        \node (d) at (-1,-1) {$\cat^\bla$};
        \draw[->] (b) to[out=20,in=160] node[above,midway]{$K$} (a);
        \draw[->] (a) to[out=-160,in=-20] node[below,midway]{$T$} (b);
         \draw[->] (d) to[out=20,in=160]
         node[above,midway]{$\mathbb{K},\mathbb{C}$}  (c);
        \draw[->] (c) to[out=-160,in=-20]
        node[below,midway]{$\mathbb{T},\mathbb{E}$} (d);
        \draw[->] (c) --node[right,midway]{$\Xi$} (a) ;
         \draw[->] (d) -- node[left,midway]{$\Xi$} (b);
    \end{tikzpicture}
\end{equation*}
\end{prop}
\begin{proof}
We need only check this for $K$, since in both cases, the functors
above are in adjoint pairs.  

Using the compatibility results for functors proved in
\cite[4.9 \& 4.10]{WebCTP}, we can reduce to the case
where the cup is added at the far right. Let  $\fl$ is be the standard Levi of type $(N-n,n)$. In this case, the
ES-equivalence is just given by $\ind_{\fl}^{\fgl_N}(-\otimes \C^n)$, since this sends  $\operatorname{pr}_n(P^\kappa_d)$ to
$\operatorname{pr}_n(Q^\kappa_d)$.  On the other
hand, we already know by \cite[4.10]{WebCTP} that this is intertwined with
$\mathbb{S}^{\bla,(\om_1,\om_{n-1})}(-,L_{\om_1})$, which matches with
$\mathbb{K}$ as shown in the proof of Proposition \ref{qt-cat}.
\end{proof}

These propositions show that our work matches with that of Sussan
\cite{Sussan2007} and Mazorchuk-Stroppel \cite{MS09}, though the
latter paper is ``Koszul dual'' to our approach above.   Recall that
each block of $\tcO_n$ has a Koszul dual, which is also a block of
parabolic category $\cO$ for $\mathfrak{gl}_N$ (see \cite{Back99}).
In particular, we have a Koszul duality
equivalence $$\Ko:\Dbe(\tcO_n^\fp)\to
D^{\downarrow}({^{n}_{\fp}\tcO})$$ where ${^n_\fp\tcO}$ is the direct
sum over all $n$ part compositions $\mu$ (where we allow parts of size
0) of a block of $\fp_\mu$-parabolic category $\tcO$ for
$\mathfrak{gl}_N$ with a particular central character depending on 
$\fp$.

Now, let $T$ be an oriented tangle labeled with $\bla$ at the bottom and $\bla'$ at top, with all appearing labels being fundamental.  Then, as before, associated to $\bla$ and $\bla$ we have parabolics $\fp$ and $\fp'$.
\begin{prop}
Assume $\bla$ only uses the fundamental weights $\om_1$ and $\om_{n-1}$. Then we have a commutative diagram
\begin{equation*}
    \begin{tikzpicture}[yscale=.9,xscale=2.4,very thick]
        \node (e) at (1,3) {$D^{\downarrow}({^n_{\fp'} \tcO})$};
        \node (f) at (-1,3) {$D^{\downarrow}({^n_{\fp} \tcO})$};
        \node (a) at (1,1) {$\Dbe(\tcO_n^{\fp'} )$};
        \node (b) at (-1,1) {$\Dbe(\tcO_n^{\fp})$};
        \node (c) at (1,-1) {$\cat^{\bla'}$};
        \node (d) at (-1,-1) {$\cat^{\bla}$};
        \draw[->] (f) -- (e) node[above,midway]{$\mathcal{F}(T)$};
	        \draw[->] (b) -- (a) node[above,midway]{$\mathbb{F}(T)$};
        \draw[->] (d) -- (c) node[above,midway]{$\Phi(T)$};
        \draw[->] (c) --(a) node[right,midway]{$\Xi$};
         \draw[->] (d) --(b) node[left,midway]{$\Xi$};
        \draw[->] (a) --(e) node[right,midway]{$\Ko$};
         \draw[->] (b) --(f) node[left,midway]{$\Ko$};
    \end{tikzpicture}
\end{equation*}
where $\mathbb{F}(T)$ is the functor for a tangle defined by Sussan in \cite{Sussan2007} and $\mathcal{F}(T)$ is the functor defined by Mazorchuk and Stroppel in \cite{MS09}.

  Our invariant $\EuScript{K}$ thus coincides with the knot invariants
  of both the above papers when  all components are labeled
  with the defining representation.  In particular,  it coincides with Khovanov homology when
  $\fg=\mathfrak{sl}_2$ and Khovanov-Rozansky homology when
  $\fg=\mathfrak{sl}_3$.
\end{prop}
\begin{proof}
We need only check that we define the same functors as Sussan and
Mazorchuk-Stroppel on a single crossing of strands labeled $\om_1$ and
on cups and caps.  
In  \cite[\S 6]{Sussan2007}, the action of crossings is given by twisting
functors and in \cite[\S 6]{MS09} by shuffling functors; thus, Proposition
\ref{trans-braid} identifies our crossing with Sussan's and the
duality of twisting and shuffling functors proven in \cite{RH} shows
that it matches that of Mazorchuk and Stroppel.

Since Sussan's cup and cap functors defined in \cite[\S 3.2]{Sussan2007} are defined by applying a
Zuckerman functor after the ES-equivalence $\cO^\fp_n\cong
\cO^{\fp'}_n$ on objects, Proposition \ref{trans-cup-cap} shows that
our functors agree with his; similarly, Mazorchuk and Stroppel's
functor is an ES-equivalence Koszul dual to ours, followed by a
translation functor, which matches our Zuckerman functor by  \cite{RH}.
\end{proof}
We believe strongly that this homology agrees with that of
Khovanov-Rozansky when one uses the defining representation for all
$n$ (this is conjectured in \cite{MS09}), but actually proving this
requires an improvement in the state of understanding of the
relationship between the foam model of Mackaay, \Stosic and Vaz
\cite{MSVsln} and the model we have presented.  It would also be
desirable to compare our results to those of Cautis-Kamnitzer for
minuscule representations, and Khovanov-Rozansky for the Kauffman
polynomial, but this will require some new ideas, beyond the scope of
this paper.

 \bibliography{../gen}
\bibliographystyle{amsalpha}
\end{document}